\documentclass[
 10pt
]{amsart}
\usepackage[latin1, utf8]{inputenc}
\usepackage{amsmath}
\usepackage{amsfonts}
\usepackage{amssymb}
\usepackage{amsthm}
\usepackage{tikz-cd}
\usepackage{color}
\usepackage{enumerate}
\usepackage{enumitem}
\usepackage{graphicx}
\usepackage{epstopdf}
\usepackage{nomencl}
\usepackage{setspace}
\usepackage{thmtools}  
\usepackage{cite}
\usepackage{mathtools}

\usepackage{mathrsfs}

\newcommand{\Tau}{T}

\renewcommand{\cdot}{\bullet}

\newcommand{\C}[1]{ {C^{\cdot}(#1)}}
\newcommand{\rC}[1]{ {\widetilde{C}^{\cdot}(#1)}}

\newcommand{\Ck}[1]{ {C^{k}(#1)}}
\newcommand{\rCk}[1]{{\widetilde{C}^{k}(#1)}}
\newcommand{\Cp}[1]{ {C^{q}(#1)}}
\newcommand{\rCp}[1]{ {\widetilde{C}^{q}(#1)}}
\newcommand{\Cq}[1]{ {C^{r}(#1)}}
\newcommand{\rCq}[1]{ {\widetilde{C}^{r}(#1)}}
\newcommand{\Cpq}[1]{ {C^{q+r}(#1)}}
\newcommand{\rCpq}[1]{ {\widetilde{C}^{q+r}(#1)}}
\newcommand{\Cpqq}[1]{{{C^{q+r-1}({#1})}}}

\newcommand{\Cc}[1]{ {C_{\cdot}(#1)}}

\newcommand{\tL}{{t_{<k}L}}

\newcommand{\cL}{{cone(\tL)}}
\newcommand{\TW}{{\tau_{\geq k}\Omega^{\cdot}(L)}}
\newcommand{\tW}{{\tau_{<k}\Omega^{\cdot}(L)}}

\newcommand{\WL}{{\Omega^{\cdot}(L)}}%

\newcommand{\HH}[1]{{H^{\cdot}(#1)}}
\newcommand{\Hp}[1]{{H^{q}(#1)}}
\newcommand{\Hpp}[1]{{H^{q-1}(#1)}}
\newcommand{\Hk}[1]{{H^{k}(#1)}}

\newcommand{\I}{{I^{\bar{p}}X}}
\newcommand{\WI}{{\Omega I^{\cdot}_{\bar{p}}(\bar{X})}}
\newcommand{\WIzero}{{\Omega I^{0}_{\bar{p}}(\bar{X})}}
\newcommand{\WIp}{{\Omega I^{q}_{\bar{p}}(\bar{X})}}
\newcommand{\WIq}{{\Omega I^{r}_{\bar{p}}(\bar{X})}}

\newcommand{\WX}{{\Omega^{\cdot}(\bar{X})}}

\renewcommand{\L}{{\Tau_{\geq k} \C{L}}}

\newcommand{\Lp}{{\Tau_{\geq k} \Cp{L}}}
\newcommand{\Lq}{{\Tau_{\geq k} \Cq{L}}}
\newcommand{\Lpq}{{\Tau_{\geq k} \Cpq{L}}}

\newcommand{\Q}{{Q^{\cdot}}}

\newcommand{\Crel}{{\C{\bar{X},L}}}
\newcommand{\Rb}{{{\bar{C}}^{\cdot}}}
\newcommand{\Wbp}{\WIp}
\newcommand{\Wrelp}{{\Omega^{q}(\bar{X},L)}}

\newcommand{\Crelpqq}{{{\Cpqq{\bar{X},L}}}}

\newcommand{\Wrelq}{{\Omega^{r}(\bar{X},L)}}

\newcommand{\Wbq}{\WIq}

\newcommand{\Rbpqq}{{\bar{C}^{q+r-1}}}

\newcommand{\incl}{\text{incl}}

\newcommand{\RR}{\mathbb{R}}
\newcommand{\ZZ}{\mathbb{Z}}

\newcommand{\A}{\mathscr{A}}
\newcommand{\B}{\mathscr{B}}
\newcommand{\sS}{\mathscr{S}}

\newcommand{\J}{\mathscr{J}}

\newcommand{\id}{\text{id}}
\newcommand{\im}{\text{im}}
\newcommand{\CI}{\Rb}

\newsavebox{\overlongequation}

\declaretheorem[name=Definition, numberwithin=section]{df}
\declaretheorem[name=Proposition, numberwithin=section]{prop}
\declaretheorem[name=Lemma, sibling=prop]{lem}
\declaretheorem[name=Remark, numbered=no]{rmk}
\declaretheorem[name=Theorem]{thm}

\declaretheorem[name=Corollary, numbered=no]{cor}

\title[Multiplicative De Rham Theorems for Relative Cohom. and HI]{Multiplicative de Rham Theorems for Relative and Intersection Space Cohomology}
\author{Franz Wilhelm Schl\"oder}
\address{Department of Mathematics and its Applications, 
University of Milano-Bicocca, 
Via Cozzi 55, 
20125 Milano, 
Italy}
\email{franz.schloeder@unimib.it}

\author{J. Timo Essig}
  \address{Department of Mathematics, Faculty of Science, Hokkaido University, Sapporo 060-0810, Japan}
  \email{essig@math.sci.hokudai.ac.jp}

\date{March, 2019}

\begin{document}

\maketitle

\begin{abstract}
We construct an explicit de Rham isomorphism relating the cohomology rings of Banagl's de Rham and spatial approach to intersection space cohomology for stratified pseudomanifolds with isolated singularities. Intersection space (co-)homology is a modified (co-)homology theory extending Poincaré Duality to stratified pseudomanifolds. The novelty of our result compared to the de Rham isomorphism given previously by Banagl is, that we indeed have an isomorphism of rings and not just of graded vector spaces. 
We also provide a proof of the de Rham Theorem for cohomology rings of pairs of smooth manifolds which we use in the proof of our main result.
 \end{abstract}

\tableofcontents

\section{Introduction}
We prove that the de Rham approach to intersection space cohomology yields the same cohomology ring as the spatial approach in analogy to ordinary cohomology on smooth manifolds. We give an explicit ring isomorphism that integrates smooth forms on the top stratum over smooth cycles.

In Section \ref{sect:sheaf}, we use classical sheaf theory to prove that integration of differential forms on a smooth manifold over smooth cycles induces a ring isomorphism between the relative de Rham and singular cohomology rings.
To prove the multiplicativity with respect to a cup product 
$\cup:  H^p (M,L) \times  H^q (M,F) \rightarrow  H^{p+q} (M, L \cup F)$ induced by the wedge product of forms we need submanifolds $L,F \subset M$ which satisfy the restrictive condition that their union $L \cup F \subset M$ is also a submanifold. This is trivially fulfilled for $L=F$,though, and we use the corresponding relative de Rham result in the second part of the paper, where we prove the existence of a multiplicative de Rham isomorphism for intersection space cohomology.

Intersection space cohomology is a method, introduced by Banagl in \cite{intspaces}, to re-establish Poincar\'e duality for singular spaces by assigning a family of so-called intersection spaces $\I$ indexed by Goresky-MacPherson perversity functions $\bar{p}$ to an $n$-dimensional stratified pseudomanifold $X$. The intersection space cohomology $HI^{\cdot}_{\bar{p}}(X)$ of $X$ is defined to be the reduced singular cohomology of $\I$ with coefficients in $\mathbb{Q}$ or, as in our case, in $\mathbb{R}$. If $\bar{q}$ is the complementary perversity of $\bar{p}$, Poincar\'e duality holds in the sense that $HI^{\cdot}_{\bar{p}}(X)\cong HI^{n-\cdot}_{\bar{q}}(X)$.

The same duality statement is true for intersection cohomology introduced in \cite{GM1, GM2}. Intersection cohomology is Goresky and MacPherson's original theory to re-establish Poincar\'e duality on singular spaces. Note, that intersection cohomology and intersection space cohomology are not isomorphic but tend to be interchanged by mirror symmetry. The former can be tied up to type IIA string theory while the latter relates to type IIB.

In \cite{HIdR}, Banagl introduces a description of intersection space cohomology for pseudomanifolds of stratification depth 1 and with geometrically flat link bundle as the cohomology of a complex of smooth differential forms on the top stratum or the blowup of $X$. This enlarges the class of pseudomanifolds to which intersection space cohomology is applicable. In \cite{BanHun}, Banagl and Hunsicker give a $L^2-$description of intersection space cohomology in the case of stratification depth 1 and product link bundle.
In \cite{HIdR_Depth2} the second author uses the differential form approach to define intersection space cohomology for pseudomanifolds of stratification depth 2 with zero dimensional bottom stratum and geometrically flat link bundle for the intermediate stratum. 
De Rham theorems for intersection space cohomology are given in \cite{HIdR} for 
pseudomanifolds with isolated singularities and in \cite{EssigDeRhamThmHI} for pseudomanifolds of depth one with product link bundles. In both cases, the de Rham isomorphisms are given by integrating differential forms over certain smooth cycles. 

A description of intersection cohomology via smooth differential forms was provided in \cite{Bry}. A different approach to intersection cohomology is pursued by Brasselet and Legrand in \cite{BrasseletLegrand1} and \cite{BrasseletLegrand2}, using a complex of differential forms with coefficients in the module of poles. De Rham theorems similar to the ones for intersection space cohomology are given by Brasselet, Hector and Saralegi in \cite{Sara1} and \cite{Sara2}.  

In contrast to intersection cohomology, both approaches to intersection space cohomology naturally come with a perversity internal cup product. Neither of the above de Rham theorems clarifies whether the constructed isomorphisms respect this multiplicative structure.
This is the topic of the main part of this paper. We establish an isomorphism of the cohomology rings in the case of isolated singularities.

As an application of our result, note that intersection space cohomology provides the correct count of massless 3-branes in type IIB string theory on a conifold \cite{intspaces}. The de Rham description allows to represent those branes as differential forms and our result represents the intersection product of the branes as the wedge product of these forms.

For a space $X'$ with only isolated singularities the intersection space cohomology coincides by construction with the intersection space cohomology of the space $X$ obtained by collapsing all the singularities into a single one. Therefore we only consider the case of one isolated singularity and can think of a stratified pseudomanifold $X$ of dimension $n$ as 
\[
X=cone(i_\partial):=(\bar{X}\cup cone(L))/\sim  ~.
\]
Here $\bar{X}$ is a smooth manifold of dimension $n$ with boundary $L$ and $i_\partial$ the inclusion of this boundary. The relation ``$\sim$'' glues the bottom of the cone to the boundary of $\bar{X},$ identifying the cone coordinate with the collar coordinate of a smooth collar of the boundary. In the more general context, $\bar{X}$ is the blowup of the singular space $X$ and $L$ is the link of the singularity. Let us briefly describe the two approaches to intersection space cohomology.

The spatial approach uses Moore approximation to truncate the links. This technique is also referred to as spatial homology truncation in \cite{intspaces} and is Eckmann-Hilton dual to Postnikov approximation. In this process we associate to the link $L$ its degree $k$ spatial (co-)homology truncation $\tL$ by homotopy theoretic methods. The space $\tL$ is a $k$ dimensional CW complex with (co)homology groups isomorphic to that of $L$ in degrees smaller than $k$ and zero otherwise. The (co)homology isomorphisms in degrees smaller than $k$ are induced by a continuous map $f:\tL \to L$. 
In \cite{intspaces}, Banagl proves that such a (co)homology truncation, together with the described map $f$, exists if the link is a simply connected CW complex and $k\geq 1$. This construction involves a choice of a splitting of the boundary map $\partial_k :C_k(L) \to im(\partial_k)$, where $C_\bullet$ here and in the rest of the paper denotes the cellular chains and cellular cochains are written as $C^\bullet$ analogously. Importantly the intersection space cohomology, constructed in this way is independent of the choice of the splitting. The intersection space is defined as the homotopy cofiber of the composition
\(
g:=i_\partial \circ f ,
\)
with $i_\partial: \partial \bar{X} \hookrightarrow \bar{X}$ the inclusion of the boundary, i.e.
\[
\I:=cone(g)=(\bar{X}\cup cone(\tL))/\sim ~,
\]
where ``$\sim$'' glues the bottom of the cone to the boundary by $g$ and $k=n-1-\bar{p}(n)$. Note that due to the restrictions on the values of the Goresky-MacPherson perversity function $\bar{p}$ we have that $k\geq 1$ and assume this for the rest of the paper. As explained above the perversity $\bar{p}$ intersection space cohomology of $X$ is defined as the reduced singular cohomology of the intersection space, i.e.
\[
HI^{\cdot}_{\bar{p}}(X):=\widetilde{H}^\bullet \left( \I \right)~,
\]
and has a ring structure given by the cup product of $\widetilde{H}^\bullet \left( \I \right) $.
 
The de Rham approach to intersection space cohomology uses a complex of forms on the top stratum or the blowup of the pseudomanifold. If we fix a Riemannian metric on $L$ we can define the degree $k$ cohomology cotruncation of $\WL$ as a subcomplex by setting
\[
\TW:= \begin{cases}
0, &\text{for }\cdot<k\\
ker(d^*), &\text{for }\cdot=k\\
\WL,& \text{for }\cdot \geq k
\end{cases}
\]
where $d^*$ is the Hodge dual of the differential of $\WL$ in degree $k-1$. The choice of the metric does not affect the cohomology groups we obtain, as demonstrated in \cite{HIdR}. $\TW$ cotruncates the cohomology of $\WL$ in the sense that
the subcomplex inclusion induces an isomorphism on cohomology in degrees $\geq k$ whereas the cohomology of $\TW$ is zero in degrees smaller than $k.$ 

Rather than the original definition as in \cite{HIdR}, we adopt the definition in \cite{BanHun} and set
 \[
 \WI:=\{\omega\in \WX| i_\partial^{\#}(\omega)\in \TW\}~
 \]
 where $k=n-1-\bar{p}(n)$ as above and $i_\partial^{\#}$ denotes the pullback of differential forms along $i_\partial$. We use $\#$ to indicate both pullbacks of differential forms and induced maps on cellular cochain complexes. In practice this should not lead to any confusion and we reserve the notation $i_\partial^*$ for the induced map on cohomology. This distinction is more relevant in our work. Note, that $\WI$ with the restricted wedge product is a sub-DGA of $\WX$. This product turns $\HH{\WI}$ into a ring. The final result of this paper is
\begin{restatable*}[Multiplicative $\WI$ de Rham Theorem]{thm}{Main}
The cohomology rings $\HH{\WI}$ and  $\HH{\rC{\I}}$ are isomorphic.
\end{restatable*}

To show this, we construct a de Rham map $\phi,$ which is different from the one provided by Banagl in \cite{HIdR} on cochain level. However, both are related on cohomology as we prove in Section \ref{Compatibility}. Observe, that $\I$ is a pushout by construction. In Section \ref{sect:pullback}, we establish that in the category of cochain complexes, the reduced cochain complex of $\I$ fits into the pullback diagram
\[
\begin{tikzcd}
\rC{\I} \dar[dashed] \rar[dashed] &\rC{cone(\tL)} \dar{i_0^{\tilde{\#}}}\\
\C{\bar{X}} \rar{g^{\#}} & \C{\tL}.
\end{tikzcd}
\]
This is not true in the category of differential graded algebras, though, since the standard quasi-isomorphism between the algebraic cellular mapping cone of the map $g$ and the reduced cellular cochain complex of the topological cone of $g$ is not a DGA-morphism. To bypass this problem, we show that the isomorphism $\phi_1$ between $\rC{\I}$ and the pullback of the above diagram induces a ring isomorphism on cohomology. The proof worked out in Section \ref{sect:phi1multqis}, is based on comparing the cohomology rings of the topological mapping cone of $g$, its mapping cylinder and the cohomology of $\ker (g^{\#})$.
We then use the universal property of the pullback to construct our intersection space de Rham map.

We next describe the maps involved in this construction. In Section \ref{cellulardR}, we combine the classical de Rham map with several other constructions to get a map $\WX \to \C{\bar{X}}$ that restricts to a map $\WI \to \C{\bar{X}}$. On the other hand, Lemma \ref{gamma} provides us with a map that, in combination with the aforementioned de Rham map and the map that is induced by the inclusion of the boundary of $\bar{X}$, yields a map $\WI \to \widetilde{C}^\bullet(cone(\tL))$. The construction of Lemma \ref{gamma} heavily uses the fact that we map from an object that is cotruncated to degree $k$ to $\C{cone(\tL)}$ which is essentially truncated to degree $k+1$. The latter property forces us to work with cellular cochains on the spatial side.

A 5-Lemma argument establishes that our map indeed is a quasi-isomorphism. However, a difficulty arises because the de Rham map only becomes multiplicative at cohomology level. If we had multiplicativity on cochain level, the pullback construction would have been in the category of DGAs, the constructed map would have been a DGA homomorphism and accordingly would have induced a multiplicative map on cohomology, too. Our strategy to deal with this problem is to factorize the intersection space de Rham map $\phi$ into a part that is a DGA homomorphism, a map $\tilde{\rho}$ that sits between the absolute and relative de Rham map and the isomorphism between $\C{\I}$ and the true pullback in the diagram above. The maps induced by the DGA homomorphism is already multiplicative on representative level and we check the multiplicativity of $\tilde{\rho}$ on cohomology explicitly by using the results of the first part of this paper. As mentioned before we also establish that the isomorphism between $\C{\I}$ and the true pullback above is multiplicative on cohomology.

\section{A Multiplicative Relative de Rham Theorem}\label{sect:sheaf}
In this section, we introduce relative de Rham cohomology groups via sheaf cohomology and then prove that the multiplicative de Rham isomorphism between absolute de Rham and singular cohomology groups descends to a multiplicative isomorphism between relative groups. This fact is then used to prove that there is a multiplicative de Rham isomorphism between spatial and de Rham description of intersection space cohomology.
\subsection{Sheaf Theory}
We use sheaf cohomology to prove a result about ordinary relative singular and de Rham cohomology. Basics about sheaves and sheaf cohomology can be found in \cite{Bredon_ST}. We recall only the notion of supports:
\begin{df}(see \cite[Def. I-6.1]{Bredon_ST})\\
Let $X$ be a topological space. A family of supports on $X$ is a family $\Phi$ of closed subsets of $X$ such that
\begin{enumerate}
\item A closed subset of an element of $\Phi$ is an element of $\Phi$;
\item $\Phi$ is closed under finite unions.
\end{enumerate}
$\Phi$ is a paracompactifying family of supports if in addition
\begin{enumerate}[resume]
\item each element of $\Phi$ is paracompact.
\item each element of $\Phi$ has a (closed) neighbourhood also contained in $\Phi.$
\end{enumerate}
\end{df}
Examples of supports are the family of all closed subsets of $X$, and the family consisting of the empty set. The first is paracompactifying if $X$ is paracompact. If $s \in \A (X)$ is a global section of a sheaf $\A$ on $X$, then $|s| = \{ x \in X | s(x) \neq 0 \}$ denotes its support. The sections of $\A$ with supports in $\Phi$ are defined by
\[
\Gamma_\Phi (\A) := \{ s \in \A (X) | ~ |s| \in \Phi \}.
\]
In the same way one defines $A_\Phi \left( X  \right) := \{ s \in \A (X) | ~ |s| \in \Phi \}$ for the presheaves of differential forms $A = \Omega^\bullet$ and singular cochains with values in some locally constant sheaf $\A$ on $X$, $A = S^\bullet (- ; \A)$. The de Rham and singular cohomology with supports in $\Phi$ is then defined by taking the cohomology groups $H^p \left( \Omega_\Phi^\bullet (X) \right)$ and $H^p \left( S^\bullet_\Phi (X; \A) \right)$. The sheaf cohomology groups with supports in $\Phi$ for the sheaf $\A$ are defined by taking any injective resolution $\A \rightarrow \J^\bullet $ of $\A $ and setting
\[
H_\Phi^r (X;\A) := H^r \left( \Gamma_\Phi (\J^\bullet) \right).
\]

\subsection{Relative Singular Cohomology}\label{subs:relsingcohom}
Before explaining the notions of relative de Rham cohomology, we recall the results of \cite[Chapter III-1]{Bredon_ST} about relative singular cohomology. For our purpose, it is sufficient to consider the reals $\RR$ as base ring for our singular cohomology groups and therefore all sheaves are sheaves of real vector spaces and all tensor products are taken over the reals. In this section, let $X $ denote an arbitrary topological space. Later, we specify $X$ to be a smooth manifold. Let $\A$ be a sheaf on $X$ and let $\Phi$ be a paracompactifying family of supports on $X.$ The singular cohomology groups of $X$ with coefficients in $\A$ and support in the paracompactifying family $\Phi$ are then defined by
\[
{}_S H_\Phi^p (X;\A) := H^p \left( \Gamma_\Phi (\sS^\bullet \otimes \A) \right),
\]
where $\sS^\bullet = \sS^\bullet (X;\RR)$ is the sheafification of singular cochains.
Note that these cohomology groups agree with the regular singular cohomology groups with real coefficients $H_S^\bullet (X;\RR)$ for $\A =\RR $ the constant sheaf and $\Phi$ the family of all closed subsets of $X.$ The ordinary, singular cup product induces a homomorphism 
\[
\cup: {}_S H_\Phi^p (X;\A) \otimes {}_S H_\Psi^q (X;\B) \rightarrow {}_S H_{\Phi \cap \Psi}^{p+q} (X; \A \otimes \B)
\]
with the usual properties (see \cite[Theorem II-7.1]{Bredon_ST}). If $X$ is \emph{HLC} (= singular homology locally connected), e.g. $X$ a manifold or more generally a CW complex, then there is a multiplicative isomorphism between sheaf cohomology and singular cohomology groups:
\begin{equation}
\theta: H_\Phi^\bullet (X;\A) \xrightarrow{\cong} {}_S H_\Phi^\bullet (X;\A)~,
\label{eq:theta_sing_sheaf}
\end{equation}
(see \cite[pp. 180-181]{Bredon_ST} for a more detailed explanation).

To define relative singular cohomology with coefficients in the sheaf $\A$, let $ F \subset X$ be a closed subspace and consider the homomorphism $ \Gamma_\Phi (\sS^\bullet (X;\RR) \otimes \A) \twoheadrightarrow \Gamma_{\Phi|F} (\sS^\bullet (F;\RR) \otimes \A|F).$ It is a surjection, since the kernel of the epimorphism $( \sS^\bullet (X;\RR) \otimes \A)|F \twoheadrightarrow \sS^\bullet (F;\RR) \otimes \A$, which is induced by a restriction morphism, is an $\sS^0 (X;\RR)|F$-module and hence $\Phi|F$-soft. Therefore, \cite[Theorem II 9.9]{Bredon_ST} is applicable.

Let $K_\Phi^\bullet (X,F;\A) $ denote the kernel of this map and define the relative singular cohomology groups of the pair $(X,F)$ with coefficients in $\A$ and supports $\Phi$ by
\[
{}_S H_\Phi^\bullet (X,F;\A) := H^\bullet \left( K_\Phi^\bullet (X,F;\A) \right).
\]
By definition, one gets the usual long exact sequence of a pair.
\[
\dots \ \rightarrow {}_S H^p_\Phi (X,F;\A) \rightarrow {}_S H^p_\Phi (X;\A)\rightarrow {}_S H^p_{\Phi|F} (F;\A|F)  \xrightarrow{+1} \ \dots
\]
Let $U_F = X -F $ denote the complement of the closed set $F$ and $\A_{U_F}$ the extension by zero to $X$ of the restriction $\A|U_F,$ see \cite[I 2.6]{Bredon_ST}.
The morphism 
\[
\Gamma_\Phi \left(  \sS^\bullet (X;R) \otimes \A_{U_F}  \right) \rightarrow K_\Phi^\bullet (X,F;\A)
\] 
induces an isomorphism on cohomology, which follows by a 5-Lemma argument (see \cite[p. 183]{Bredon_ST} for details). By this isomorphism we can introduce a relative cup product given by the following composition 
\[
\begin{tikzcd}
{}_S H_\Phi^p (X,F;\A) \otimes {}_S H_\Psi^q (X,L;\B) \ar{r}{\cup} \ar{d}{\cong} &  {}_S H_{\Phi\cap\Psi}^{p+q} (X,F\cup L; \A \otimes \B) \\
{}_S H_\Phi^p (X;\A_{U_F}) \otimes {}_S H_\Psi^q (X;\B_{U_L}) \ar{r}{\cup} & {}_S H_{\Phi\cap\Psi}^{p+q} (X, \A_{U_F} \otimes \B_{U_L}) \ar{u}{\cong}
\end{tikzcd}
\]
The vertical map on the right is induced by the inclusion $ \A_{U_F} \otimes \B_{U_L} \hookrightarrow \A \otimes \B,$ $F,L \subset X$ are closed, and $U_F = X-F, U_L = X-L.$
This coincides with the ordinary relative cup product in singular cohomology for $\A = \B = \RR$ and $\Phi, \Psi$ the families of all closed subsets of $X.$ Also, together with the long exact cohomology sequence for sheaf cohomology of pairs and the maps $\theta$ of (\ref{eq:theta_sing_sheaf}), this definition gives a multiplicative isomorphism
\[
\theta: H_\Phi^\bullet (X,F;\A) \xrightarrow{\cong} {}_S H_\Phi (X,F;\A)
\]
for $X,F$ both \emph{HLC}.

Note that if $X, F$ are smooth manifolds, one can use smooth singular cochains instead of continuous ones. To see this, let $ S^\bullet_\infty (X; \RR) $ denote the complex of smooth singular cochains with coefficients in $\RR$ and let $\rho: S^\bullet (X;\RR) \twoheadrightarrow S^\bullet_{\infty} (X;\RR) $ denote the restriction of the complex of all singular cochains to smooth ones. 

Then, one gets a map of sheaves $\rho: \sS^\bullet (X;\RR) \otimes \A \rightarrow \sS^\bullet_{\infty} (X;\RR) \otimes \A, $ which we also denote by $ \rho.$ Further, for any sheaf $ \A $ on $X$ and any family of supports $\Phi$, one gets a map $\rho: \Gamma_{\Phi} \left( \sS^\bullet (X;\RR) \otimes \A \right) \rightarrow \Gamma_{\Phi} \left( \sS^\bullet_{\infty} (X;\RR) \otimes \A \right). $ 
Note, that for $\Phi$ paracompactifying, all the sheaves $\sS^r (X;\RR)$ and $\sS^r_\infty (X;\RR)$ are $\Phi$-soft as modules over the sheaf of continuous respectively smooth real valued functions, which are $\Phi$-soft by a standard partition of unity argument. Hence, the sheaves $\sS^\bullet (X;\RR) \otimes \A$ and $\sS^\bullet_\infty (X;\RR) \otimes \A$ are resolutions of $\A$ by $\Phi$-soft sheaves and $\rho$ induces an isomorphism on cohomology groups by \cite[II-4.2]{Bredon_ST}, 
\[
\rho^*: {}_S H_\Phi^\bullet (X;\A) \xrightarrow{\cong} {}_S^{\infty} H_\Phi^\bullet (X;\A).
\]
Here, ${}_S^{\infty} H_\Phi^\bullet (X;\A) = H^\bullet \left( \Gamma_\Phi \left( \sS^\bullet_{\infty} (X;\RR) \otimes \A \right) \right).$
Let $K_{\Phi, \infty}^\bullet (X,F; \A) $ denote the kernel of the surjection 
\[
\Gamma_\Phi \left( \sS^\bullet_{\infty} (X;\RR) \otimes \A \right) \twoheadrightarrow \Gamma_{\Phi|F} \left( \sS^\bullet_{\infty} (F;\RR) \otimes \A|F \right) 
\]
and define the smooth relative singular cohomology groups with values in $\A$ by
\[
{}_S^{\infty}H_\Phi^\bullet \left( X,F; \A \right) := H^\bullet \left( K^\bullet_{\Phi, \infty} (X,F;\A) \right).
\]
Again, there is a long exact sequence of the pair $(X,F)$ and we get a cup product on the relative smooth singular cohomology groups, that coincides with the regular one for $\A = \RR $ and $\Phi $ the family of all closed subsets. 

Restriction of singular cochains to smooth chains induces a multiplicative isomorphism on cohomology as follows from the following commutative diagram
\[
\begin{tikzcd}
{}_S H_\Phi^\bullet (X;\A_{U_F}) \ar{d}{\cong}[swap]{\text{mult}} \ar{r}{\rho^*}[swap]{\cong, \text{mult}} & {}_S^\infty H_\Phi^\bullet (X;\A_{U_F}) \ar{d}{\cong}[swap]{\text{mult}} \\
{}_S H_\Phi^\bullet (X,F;\A) \ar{r}{\rho^*} & {}_S^\infty H_\Phi^\bullet (X,F;\A)~,
\end{tikzcd}
\]
where the vertical map on the right is a multiplicative isomorphism analogously to the non-smooth case.

\subsection{Relative de Rham Cohomology}
To consider de Rham cohomology, we need smooth manifolds. We prove a relative version of de Rham's Theorem for the following pairs of smooth manifolds (possibly with boundary). Let \( M^n \) be a smooth manifold and \( F^m \subset M \) a submanifold of dimension \( m \) which is closed as a subspace (not necessarily as a manifold). The pair $(M,F)$ might be compact (or $M$ open and $F$ compact, or both non-compact manifolds). We only consider submanifolds that are closed subsets, since then the relative sheaf cohomology groups can be replaced by absolute cohomology groups of the complement.

If two different submanifolds $F^m, L^m \subset M$ occur, we demand that their union $F \cup L \subset M$ is also a submanifold. This is relevant later to insure that relative cup products on de Rham cohomology actually have a well-defined target.

Let  \( \A \) be a sheaf of  \( \RR - \)modules on \( M \). 
The de Rham presheaves on $M$ are given by the assignments $U \mapsto \Omega^r (U)$, where $\Omega^r (U)$ is the set of smooth differential $r$-forms on the open set $U$ (of $M$ respectively $F$). This gives conjunctive monopresheaves and hence sheaves. Let $\Omega^r (M)$ denote the so defined sheaf on $M$ and $\Omega^r (F)$ the corresponding sheaf on the manifold $F$. In contrast, $\Omega^r (M)|F$ denotes the restriction of the sheaf $\Omega^r (M)$ to the subspace $F$.
The de Rham cohomology with coefficients in \( \A \) is defined as
\[ 
_{\Omega}H_{\Phi}^\bullet (M ; \A) := H^\bullet \left ( \Gamma_\Phi (\Omega^\bullet \otimes \A ) \right ).
\]
The wedge product \( \wedge: \Omega^p (U) \otimes \Omega^q (U) \rightarrow \Omega^{p+q} (U), ~ U \subset M \) open, induces a cup product on \( _\Omega H^\bullet_\Phi (M;\A). \) 
To define the relative de Rham cohomology groups, we note that the restriction homomorphism \( i^*: \Omega^\bullet (U) \rightarrow \Omega^\bullet ( U \cap F ) \) is surjective
and hence induces an epimorphism \( \left ( \Omega^\bullet (M) \otimes \A \right )|F \twoheadrightarrow \Omega^\bullet (F) \otimes \A|F \) of sheaves on \( F \). Since the kernel of this homomorphism is an \( \Omega^0 (M)|F- \)module and hence \( \Phi|F-\)soft by \cite[Theorem II 9.16]{Bredon_ST}, we get an epimorphism
\[
\Gamma_\Phi \left ( \Omega^\bullet (M) \otimes \A_F \right) = \Gamma_{\Phi|F} \left ( (\Omega^\bullet (M) \otimes \A)|F \right)  \twoheadrightarrow \Gamma_{\Phi|F} \left ( \Omega^\bullet (F) \otimes \A|F \right)
\]
of chain complexes by \cite[Theorem II 9.9]{Bredon_ST}. The kernel of $\Omega^\bullet (M) \otimes \A \to \Omega^\bullet (M) \otimes \A_F$ is $\Omega^\bullet (M) \otimes A_U,$ a $\Phi$-soft sheaf, because of \cite[II 9.18]{Bredon_ST} and the fact that $\Omega^\bullet (M)$ is $\Phi$-fine and hence $\Phi$-soft (for $\Phi$ paracompactifying). Then, again by \cite[Theorem II 9.9]{Bredon_ST}, the map \( \Gamma_\Phi ( \Omega^\bullet (M) \otimes \A) \rightarrow \Gamma_\Phi ( \Omega^\bullet (M)  \otimes \A_F) \) is also onto. Both epimorphisms combine to an epimorphism
\[
\Gamma_\Phi ( \Omega^\bullet (M) \otimes \A) \twoheadrightarrow \Gamma_{\Phi|F} (\Omega^\bullet (F) \otimes \A|F).
\] 
We let \( Q_\Phi^\bullet (M,F;\A) \) denote the kernel of this epimorphism and define the relative de Rham cohomology with coefficients in the sheaf \( \A \) as follows.
\begin{df}[Relative de Rham Cohomology]
	The relative de Rham cohomology of the pair \( (M,F) \) of smooth manifolds, $F \subset M$ closed as a subset, with coefficients in the sheaf \( \A \), is defined by
	\[
	_\Omega H_\Phi^\bullet (M,F;\A) := H^\bullet \left ( Q_\Phi^\bullet (M,F;\A) \right).
	\] 
\end{df}

As for the relative singular cohomology groups, we want to relate these groups to the absolute groups and the sheaf-theoretic cohomology groups. To do so, we note that for \( \Phi \) paracompactifying \( \Omega^\bullet \otimes \A \) is a resolution of \( \A \) by \( \Phi-\)fine sheaves and hence there is a natural isomorphism
\[
\rho: {}_\Omega H^\bullet_\Phi (M;\A) \rightarrow H^\bullet_\Phi (M;\A),
\]
which preserves cup products (see \cite[II-5.15 and II-7.1]{Bredon_ST} for details). If \( f \in C^\infty (F,M) \) is a smooth map between the smooth manifolds \( F,M \) and \( \Phi, \Psi \) are paracompactifying families  on \( M \) and \( F \) respectively, such that \( f^{-1} \Psi \subset \Phi \) and \( \A \) is a sheaf on \( F \), we get a commutative diagram
\begin{equation}
\label{diag:pullback}
\begin{tikzcd}
	_\Omega H_\Phi^\bullet (M;\A) \ar{r} \ar{d}{f^*} & H^\bullet_\Phi (M; \A) \ar{d}{f^*} \\
	_\Omega H_\Psi^\bullet (F;f^* \A) \ar{r}  & H^\bullet_\Psi (F; f^* \A)
\end{tikzcd}
\end{equation} 
in analogy to singular cohomology (compare to the similar diagram for singular cohomology on \cite[p. 182]{Bredon_ST}).

\begin{lem} \label{cor:jstar_isom}
Let \( j: F \hookrightarrow M \) be the inclusion of the submanifold \( F \subset M \), which is closed as a subspace. Let \( \Phi \) be a paracompactifying family of supports on \( M \) and let \( \A \) be a sheaf on \( M \). Then the map
\[
j^*: {}_\Omega H_\Phi^\bullet (M; \A_F) \rightarrow {}_\Omega H_{\Phi|F}^\bullet (F;\A|F)
\]
is an isomorphism and preserves cup products.
\end{lem}
\begin{proof} 
For arbitrary smooth manifolds \( M \) and arbitrary sheaves \( \A \) on \( M \) the map \( \rho: {}_\Omega H^\bullet_\Phi (M;\A) \rightarrow H^\bullet_\Phi (M;\A) \) is an isomorphism that preserves cup products. Since Diagram (\ref{diag:pullback}) commutes, it suffices to show that the map \( j^*: H_\Phi^\bullet (M; \A_F) \rightarrow H_{\Phi|F}^\bullet (F;\A|F) \) is an isomorphism of sheaf cohomology groups and preserves cup products. But this is essentially the statement of \cite[Corollary II-10.2]{Bredon_ST}.
\end{proof}

The exact sequence \( 0 \rightarrow \A_{U_F} \rightarrow \A \rightarrow \A_F \rightarrow 0 \) of sheaves induces an exact sequence 
\( 0 \rightarrow \Gamma_\Phi \left (\Omega^\bullet (X) \otimes \A_{U_F} \right ) \rightarrow \Gamma_\Phi \left (\Omega^\bullet (X) \otimes \A \right ) \rightarrow \Gamma_\Phi \left (\Omega^\bullet (X) \otimes \A_F \right ). \) Hence, the first morphism gives rise to a (trivially multiplicative) map \[ \Gamma_\Phi \left (\Omega^\bullet (X) \otimes \A_{U_F} \right ) \rightarrow Q_\Phi^\bullet (M,F; \A). \]
Together with the (multiplicative) map \[ j^*: \Gamma_\Phi \left (\Omega^\bullet (X) \otimes \A_F \right ) \rightarrow \Gamma_{\Phi|F} \left ( \Omega^\bullet (F) \otimes \A|F \right ), \]induced by the submanifold inclusion \( j: F \hookrightarrow M, \) we get a commutative diagram
\[ \begin{tikzcd}[column sep=small]
\Gamma_\Phi \left ( \Omega^\bullet (M) \otimes \A_{U_F} \right) \ar[hook]{r} \ar{d} & \Gamma_\Phi \left ( \Omega^\bullet (M) \otimes \A \right) \ar[two heads]{r} \ar{d}{=} & \Gamma_\Phi \left ( \Omega^\bullet (M) \otimes \A_F \right) \ar{d}{j^*} \\
Q_\Phi^\bullet (M,F; \A) \ar[hook]{r} & \Gamma_\Phi \left ( \Omega^\bullet (M) \otimes \A \right) \ar[two heads]{r} & \Gamma_{\Phi|F} \left ( \Omega^\bullet (F) \otimes \A|F \right) 
\end{tikzcd}
\]
We consider the induced diagram on cohomology:
\[ \begin{tikzcd}[column sep=small]
\dots \ar{r} & {}_\Omega H_\Phi^p \left ( M; \A_{U_F} \right) \ar{r} \ar{d} & {}_\Omega H_\Phi^p \left ( M; \A \right)  \ar{r} \ar{d}{=} & {}_\Omega H_\Phi^p \left ( M; \A_F \right) \ar{r}{+1} \ar{d}{j^*}[swap]{\cong} & \dots \\
\dots \ar{r} &{}_\Omega H_\Phi^p (M,F; \A) \ar{r} & {}_\Omega H_\Phi^p \left ( M; \A \right) \ar{r} & {}_\Omega H_{\Phi|F}^p \left ( F; \A|F \right) \ar{r}{+1} & \dots
\end{tikzcd}
\]
By the statement of the last lemma,
the last vertical map is an isomorphism. The 5-Lemma implies that the first vertical map is also an isomorphism,
which leaves us with the following result.
\begin{prop}\label{prop:relDeRhamtechnicalProp}
Let $(M,F)$ be any pair of smooth manifolds, possibly with boundary, where $F \subset M$ is closed as a subset, and let $ \Phi$ be any paracompactifying family of supports and $\A$ any sheaf. Then, there is an isomorphism
\[
{}_\Omega H_\Phi^p \left ( M; \A_{U_F} \right) \xrightarrow{\cong} {}_\Omega H_\Phi^p (M,F; \A).
\]
In particular, this induces a multiplicative structure on ${}_\Omega H_\Phi^p (M,F; \A), $ which coincides with the multiplicative structure induced by the relative wedge product 
\[ \wedge: \Omega^{p} (U,U\cap L) ~\otimes ~\Omega^q \left( U, U \cap F \right) \rightarrow \Omega^{p+q} \left( U, U \cap (L \cup F) \right) \] 
for $ \A = \RR,$ provided $L,F \subset M$ are two smooth submanifolds that are closed as subsets and such that $L \cup F \subset M$ is also a smooth submanifold.
\end{prop}
\begin{proof}
	What is left is a proof for the last part of the statement. So let $\A = \RR$ and $L,F \subset M$ be as in the proposition. Then the map $ \Gamma_{\Phi} \left( \Omega^\bullet (M) \otimes \A \right) \twoheadrightarrow \Gamma_{\Phi|F} \left( \Omega^\bullet (F) \otimes \A|F  \right)$ coincides with the pullback map $ j^*: \Omega^\bullet_{\Phi} (M) \twoheadrightarrow \Omega^\bullet_{\Phi|F} (F).$ The kernel of this map is the complex $\Omega^\bullet_{\Phi} (M,F) $ of the forms on $M$ that vanish on $F.$ The maps $ \RR_{U_F} \rightarrow \RR$ and $\RR_{U_L} \rightarrow \RR$ of sheaves on $M$ induce maps of complexes of sheaves $\Omega^\bullet (M) \otimes \RR_{U_L} \rightarrow \Omega^\bullet (M) \otimes \RR$ and the same with $F$ instead of $L.$ These maps induce chain maps $\Gamma_\Phi \left ( \Omega^\bullet (M) \otimes \RR_{U_F} \right ) \rightarrow \Gamma_{\Phi} \left( \Omega^\bullet (M) \otimes \RR \right) \cong \Omega_{\Phi}^\bullet (M) $ and the same map for $L$ instead of $F$. Since they factor through $ \Omega_{\Phi}^\bullet ( M, F )$, respectively $\Omega_{\Phi}^{\bullet} (M,L), $ and the differential of these complexes comes from the differential on the total de Rham complexes, we get the following induced maps.
\[
\begin{split}
H^0 \left( \Gamma_{\Phi} \left( \Omega^\bullet (M) \otimes \RR_{U_F}  \right) \right) = \ker d & \rightarrow \ker d|\Omega^\bullet_\Phi (M,F) = {}_\Omega H^0_{\Phi} (M,F),\\
H^0 \left( \Gamma_{\Phi} \left( \Omega^\bullet (M) \otimes \RR_{U_L}  \right) \right) = \ker d & \rightarrow \ker d|\Omega^\bullet_\Phi (M,L) = {}_\Omega H^0_{\Phi} (M,L).
\end{split}
\]
Since the cup product 
\[ \cup: {}_\Omega H_\Phi^p (M; \RR_{U_F}) \otimes {}_\Omega H_\Phi^q (M, \RR_{U_L} ) \rightarrow H_\Phi^{p+q} (M, \RR_{U_F}\otimes \RR_{U_L}) \] 
is induced by the wedge product of forms, we get a commutative diagram
\[
\begin{tikzcd}
{}_\Omega H^0_\Phi (M; \RR_{U_F} ) \otimes {}_\Omega H^0_\Phi (M; \RR_{U_L} ) \ar{r}{\cup} \ar{d}{\rho} & {}_\Omega H^0_\Phi (M; \RR_{U_F} \otimes \RR_{U_L} ) \ar{d}{\rho} \\
{}_\Omega H^0_{\Phi} (M,F) \otimes {}_\Omega H^0_{\Phi} (M,L) \ar{r}{\cup} & {}_\Omega H^0_{\Phi} (M,F \cup L).
\end{tikzcd}
\]
This allows us to apply \cite[Theorem II-6.2]{Bredon_ST} to the two natural transformations $\alpha \otimes \beta \mapsto \rho \left( \alpha \cup \beta \right)$ and $\alpha \otimes \beta \mapsto \rho (\alpha) \cup \rho (\beta)$ from the top left corner of the diagram to the bottom right corner. 
\end{proof}

\subsection{Relative de Rham Map}
\label{subs:rel_derham_map}
In \cite[Chapter III-3]{Bredon_ST}, Bredon proves that the classical de Rham map $k: \Omega^\bullet (M) \rightarrow S_\infty^\bullet (M;\RR), $ defined by integrating forms over smooth chains, induces a multiplicative homomorphism
\[
k^*: {}_{\Omega}H_\Phi^\bullet \left( M; \A \right) \rightarrow {}_S^\infty H_\Phi^\bullet (M;\A),
\]
which is an isomorphism for $\Phi$ paracompactifying and coincides with the usual de Rham isomorphism for $\A = \RR.$ Let $ F \subset M$ be a smooth submanifold, closed as a subspace as above, and let $j: F \hookrightarrow M.$ For each open set $U \subset M$ we get a commutative diagram
\[
\begin{tikzcd}
\Omega^\bullet (U) \ar{r}{j|_U^{\#}} \ar{d}{k_M} & \Omega^\bullet (U \cap F) \ar{d}{k_F} \\
 S_\infty^\bullet (U;\RR) \ar{r}{j|_U^{\#}} & S_\infty^\bullet (U\cap F;\RR)
\end{tikzcd}
\]
where all the maps commute with the corresponding restriction maps. Hence, for any sheaf $\A$ on $M$ this induces a commutative diagram
\[
\begin{tikzcd}
\Gamma_\Phi \left( \Omega^\bullet \otimes \A \right) \ar{d}{k_M} \ar[two heads]{r} & \Gamma_{\Phi|F} \left( \Omega^\bullet \otimes \A|F \right) \ar{d}{k_F}   \\
\Gamma_\Phi \left( \sS_\infty^\bullet (M;\RR) \otimes \A \right)\ar[two heads]{r}& \Gamma_{\Phi|F} \left( \sS_\infty^\bullet (F;\RR) \otimes \A|F \right)~.
\end{tikzcd}
\]
This implies that $k_M$ factors through the kernels of the horizontal maps. That means it induces a relative de Rham morphism
\begin{equation}
k: Q_\Phi \left( M,F; \A \right) \rightarrow K_{\Phi,\infty} \left( M, F ;\A \right).
\label{eq:rel_derham_map}
\end{equation}
\begin{thm}[Relative de Rham Theorem] \label{thm:rel_dR_mult}
	Let $(M,F)$ be a pair of smooth manifolds, $F \subset M$ closed as a subset, let $\A$ be a sheaf on $M$ and $\Phi$ a paracompactifying family of supports and let $k:Q_\Phi \left( M,F; \A \right) \rightarrow K_{\Phi, \infty} \left( M, F ;\A \right) $ denote the relative de Rham map. Then, the induced map on cohomology
 \[
k^*: {}_\Omega H_\Phi^\bullet (M,F; \A) \rightarrow {}_S^\infty H_\Phi^\bullet (M,F;\A)
\]
is a multiplicative isomorphism.
\end{thm}
\begin{proof}
As before, let $ U_F := M-F \subset M,$ which is an open subset of $M.$ We prove that the following diagram commutes.
\begin{equation} \label{relDRDiagram}
\begin{tikzcd}
{}_\Omega H_\Phi^\bullet (M; \A_{U_F} ) \ar{r}{k}[swap]{\cong} \ar{d}{\cong} & {}_S^\infty H_\Phi^\bullet (M; \A_{U_F} ) \ar{d}{\cong} \\
{}_\Omega H_\Phi^\bullet (M, F ; \A ) \ar{r}{k} & {}_S^\infty H_\Phi^\bullet (M,F; \A)
\end{tikzcd}
\end{equation}
This will complete the proof, since the vertical maps are clearly multiplicative and the isomorphism $k$ on the top is multiplicative by \cite[Theorem III 3.1]{Bredon_ST}.

Since the absolute de Rham morphism $k: \Gamma_{\Phi} \left( \Omega^\bullet (M) \otimes \A \right) \rightarrow \Gamma_\Phi \left( \sS_\infty^\bullet (M;\RR) \otimes \A \right)$ is natural, the following diagram commutes.
\[
\begin{tikzcd}
\Gamma_{\Phi} \left( \Omega^\bullet (M) \otimes \A_{U_F} \right) \ar{r}{k} \ar{d}{\tau} & \Gamma_\Phi \left( \sS_\infty^\bullet (M;\RR) \otimes \A_{U_F} \right) \ar{d}{\tau} \\
\Gamma_{\Phi} \left( \Omega^\bullet (M) \otimes \A \right) \ar{r}{k} & \Gamma_\Phi \left( \sS_\infty^\bullet (M;\RR) \otimes \A \right).
\end{tikzcd}
\]
By definition of the relative de Rham morphism, the following diagram also commutes.
\[
\begin{tikzcd}
	Q_{\Phi}^\bullet \left( M,F;\A \right) \ar{r}{k} \ar{d}{\sigma} & K_{\Phi, \infty}^\bullet \left( M,F; \A\right) \ar{d}{\sigma} \\
\Gamma_{\Phi} \left( \Omega^\bullet (M) \otimes \A \right) \ar{r}{k} & \Gamma_\Phi \left( \sS_\infty^\bullet (M;\RR) \otimes \A \right).
\end{tikzcd}
\]
The diagrams above can be combined as follows.
\[
\begin{tikzcd}
\Gamma_{\Phi} \left( \Omega^\bullet (M) \otimes \A_{U_F} \right) \ar{r}{k} \ar[bend right=70]{dd}[swap]{\tau} \ar{d} & \Gamma_\Phi \left( \sS_\infty^\bullet (M;\RR) \otimes \A_{U_F} \right) \ar{d} \ar[bend left=70]{dd}{\tau} \\
Q_{\Phi}^\bullet \left( M,F;\A \right) \ar{r}{k} \ar{d}{\sigma} & K_{\Phi, \infty}^\bullet \left( M,F; \A\right) \ar{d}{\sigma} \\
\Gamma_{\Phi} \left( \Omega^\bullet (M) \otimes \A \right) \ar{r}{k} & \Gamma_\Phi \left( \sS_\infty^\bullet (M;\RR) \otimes \A \right).
\end{tikzcd}
\]
By the previous statements, the bottom square and the exterior big square commute.
Since the $\sigma$'s are subcomplex inclusions by definition, this gives the desired commutativity of the top square, which induces Diagram (\ref{relDRDiagram}) on cohomology.
\end{proof}

\begin{cor}
	The composition of the classical relative de Rham map $k: \Omega^\bullet (M,F) \rightarrow S_\infty^\bullet (M,F;\RR), $ defined by integration of relative forms over smooth chains, and Lee's smoothing operator $s^*: S_\infty^\bullet (M,F;\RR) \to S^\bullet (M,F;\RR)$, see \cite[pp. 474 ff]{Lee}, induces a multiplicative isomorphism on cohomology.
	\[ s^* \circ k^*: H_{DR}^r (M,F) \xrightarrow{\cong} H_{S, \infty}^r (M,F;\RR) \xrightarrow{\cong} H_S^r (M,F; \RR). \]
\end{cor}
\begin{proof}
	This follows from Theorem \ref{thm:rel_dR_mult} with $\A = \RR$ and $\Phi $ the family of all closed subsets of $M$ and the results of Section \ref{subs:relsingcohom}. 
	In this setting, the relative de Rham map (\ref{eq:rel_derham_map}) coincides with the classical relative de Rham map $k: \Omega^\bullet (M,F) \rightarrow S_{\infty}^\bullet (M,F)$ defined by
\[
k (\omega) (\sigma) = \int_{\Delta_p} \sigma^{*} \omega,
\]
for $\omega \in \Omega^r (M,F) $ and any smooth $p-$simplex $\sigma$. Since Lee's smoothing operator is a chain homotopy inverse of the restriction map $\rho: S^\bullet (X;\RR) \twoheadrightarrow S^\bullet_\infty (X;\RR),$ which induces a multiplicative isomorphism on cohomology, $s^* = (\rho^*)^{-1}$ is also multiplicative on cohomology. 
\end{proof}

\section{The cohomology ring of $\I$}\label{sect:pullback}
In this section we prove
\begin{thm}\label{isopullback}
The reduced cohomology ring of $\I$ is isomorphic to the cohomology ring of the pullback $Q^{\cdot}$ in the following diagram.
\begin{equation}\label{pullbackQ}
\begin{tikzcd}
 Q^{\cdot} \arrow{r} {q_2}\arrow{d}{q_1}& \rC{\cL} \arrow{d}{i_0^{\tilde{\#}}} \\
\C{\bar{X}} \arrow{r}{g^{\#}}& \C{\tL}.
\end{tikzcd}
\end{equation}
Here $\rC{\cL}$ denotes the reduced cellular cochain complex of $\cL$, specifically realized as the cellular cochain complex relative to the cone point $\C{\cL,c}$. The map $i_0: \tL \to \cL$ denotes the inclusion of $\tL$ as the bottom of the cone, $i_0^{\tilde{\#}}: \rC{\cL} \to \C{\tL}$ denotes the composition 
\[ \rC{\cL} \hookrightarrow \C{\cL} \xrightarrow{i_0^{\#}} \C{\tL}. \] and $g: \tL \to \bar{X}$ is the map defining the intersection space.
\end{thm}
The universal property of the pullback is later used to construct the map $\phi_2 : \CI \to Q^{\cdot}$ which is the middle part of our intersection space cohomology de Rham map.

Recall that in the spatial picture of intersection space cohomology we are working with cellular cochains and accordingly the various cochain complexes of topological spaces here are their cellular cochain complexes. Note that $g$ and $i_0$ are cellular maps. Therefore they induce cochain maps, and $Q^\bullet$ is a cochain complex by construction. On the other side it is in general false that cellular maps induce multiplicative maps on cochain level and we have to clarify how the product on $Q^\bullet$ and hence $H^\bullet(Q^\bullet)$ arises.
The cup product of cellular cohomology is induced on cochain level by an Eilenberg-Zilber type map and a cellular approximation to the diagonal, called cellular diagonal approximation in the following. In Section \ref{sect:ProductQ} we demonstrate that $i_0^{\tilde{\#}}$ and $g^{\#}$ are DGA homomorphisms for the right choice of products on cochain level. This upgrades the construction above from the category of cochain complexes to the category of DGAs and $Q^{\cdot}$ is naturally equipped with an appropriate product.

After this we turn to the proof of Theorem \ref{isopullback}. An explicit quasi-isomorphism is given by the map $\phi_1$ in the diagram
\begin{equation}\label{pullbackphi1} 
\begin{tikzcd}
\rC{\I} \drar{\phi_1} \arrow{ddr}[swap]{i_{\bar{X}}^{\tilde{\#}}} \arrow{rrd}{i_{\cL}^{\#}}                        &&\\
& Q^{\cdot} \arrow{r} {q_2}\arrow{d}{q_1}& \rC{\cL} \arrow{d}{i_0^{\tilde{\#}}} \\
&\C{\bar{X}} \arrow{r}{g^{\#}}& \C{\tL}
\end{tikzcd}
\end{equation}
with $i_{\cL}: \cL \to \I$ the composition
 \[ i_{\cL}: \cL \hookrightarrow \bar{X} \sqcup \cL \twoheadrightarrow \frac{\bar{X} \sqcup \cL}{\sim} = \I \]
and $i_{\bar{X}}^{\tilde{\#}}: \rC{\I} \to \C{\bar{X}}$ the composition 
\[ \rC{\I} \hookrightarrow \C{\I} \xrightarrow{i_{\bar{X}}^{\#}} \C{\bar{X}} \]
where $i_{\bar{X}}:\bar{X} \to \I$ is the inclusion of $\bar{X}$ in $cone(g)=\I$. Note that by choosing the natural cell structures on the cones $\I$ and $\cL$ and their respective tips as base point, one can quite straight forwardly establish that the diagram commutes. Therefore, $\phi_1$ is uniquely determined by the universal property of the pullback. It is also not hard to establish that $\phi_1$ is a bijection, however it is unclear whether one can equip $\rC{\cL}$ and $\rC{\I}$ with cup products such that the map
\[
 i_{\cL}: \cL \hookrightarrow \bar{X} \sqcup \cL \twoheadrightarrow \frac{\bar{X} \sqcup \cL}{\sim} = \I 
\]
induces a DGA morphism. Therefore, we only obtain that $\phi_1$ is an isomorphism of cochain complexes and not of DGAs. Apriori it is not clear, whether $\phi_1^*$ respects the multiplication of the cohomology rings. We establish this property in section \ref{sect:phi1multqis}.

\subsection{The Product on the Pullback}\label{sect:ProductQ}
As explained above, we need to establish that we can choose products on cochain level such that $i_0^{\tilde{\#}}$ and $g^{\#}$ are DGA homomorphisms. With this choice of products, the pullback is a pullback in the category of DGAs and $Q^\bullet$ is a DGA by construction. Note that every choice of graded product on $\C{\cL}$ restricts to a product on $\rC{\cL}$ and therefore the inclusion $\rC{\cL} \hookrightarrow \C{\cL}$ is a DGA homomorphism. It is left to prove that $i_0^{\#}$ and $g^{\#}$ can be made into multiplicative maps, which is the content of Proposition \ref{inclusionDGA} and Theorem \ref{thm:gDGAhom}, respectively.

The map $i_0$ is a CW subcomplex inclusion and the multiplicativity of its induced map is covered by the following proposition.
\begin{prop}\label{inclusionDGA}
For a CW subcomplex inclusion $i: Y \xhookrightarrow{} Z$ and a given cellular diagonal approximation $\tilde{\Delta}_{Y}$ on $Y$,  we can choose a cellular diagonal approximation $\tilde{\Delta}_{Z}$ on $Z$ such that $i^{\#} : \C{Z} \to \C{Y}$ is a DGA homomorphism with respect to the cup products induced by those diagonal approximations.
\end{prop}

\begin{proof}
Choosing the canonical cell structure on the CW complexes $Y\times Y$ and $Z\times Z$ ensures that for given cells $a$ and $b$ their product $a\times b$ is a cell again. Thus the Eilenberg-Zilber type maps 
\[
 \theta_Y: \C{Y}\otimes \C{Y}  \to \C{Y \times Y}\\
\]
and
\[
 \theta_Z: \C{Z}\otimes \C{Z}  \to \C{Z \times Z}\\
\]
are given by mapping $a^\dagger \otimes b^\dagger$ bijectively to $(a\times b)^\dagger$. The daggered objects 
are the respective cochains in the basis dual to the basis of chains given by the cells. We calculate
\begin{align*}
(i \times i)^{\#} \circ \theta_Z (a^\dagger \otimes b^\dagger)&=(i \times i)^{\#} (a\times b)^\dagger = (i(a)\times i(b))^\dagger  \\ &= \theta_Y (i(a)^\dagger \otimes i(b)^\dagger) =\theta_Y \circ (i^{\#} \otimes i^{\#})   (a^\dagger \otimes b^\dagger)
\end{align*}
so
\[
(i \times i)^{\#} \circ \theta_Z  =\theta_Y \circ (i^{\#} \otimes i^{\#}).
\]
The relative version of the Cellular Approximation Theorem (cf. \cite[p.~76]{May}) 
assures that we can extend a cellular diagonal approximation $\tilde{\Delta}_Y$ of $Y$ to a cellular diagonal approximation $\tilde{\Delta}_Z$ of $Z$, i.e.
\[
  (i\times i) \circ \tilde{\Delta}_Y = \tilde{\Delta}_Z \circ i.
\]
In conclusion
\begin{align*}
i^{\#} \circ \cup_Z 	&= i^{\#} \circ \tilde{\Delta}_Z^{\#} \circ \theta_Z\\ 
					&= \tilde{\Delta}_Y^{\#} \circ (i\times i)^{\#} \circ \theta_Z \\
					&= \tilde{\Delta}_Y^{\#} \circ \theta_Y \circ (i^{\#} \otimes i^{\#})   \\
					&= \cup_Y \circ (i^{\#} \otimes i^{\#})  ~.    
\end{align*}
So, $i$ induces a DGA homomorphism on cochain level if we define the cup product via the specific diagonal approximations  $\tilde{\Delta}_Y$ and $\tilde{\Delta}_Z$.
\end{proof}
Next we establish that there are DGA-structures on $\C{\bar{X}}$ and $ \C{\tL}$ such that $g$ induces a DGA homomorphism, too.
\begin{thm}\label{thm:gDGAhom}
  Given a cellular diagonal approximation $\tilde{\Delta}_{t_{<k} L}$ on $\tL,$ we can choose a cellular diagonal approximation on $\bar{X}$ such that $g^{\#}: \C{\bar{X}} \to \C{\tL}$ becomes a DGA homomorphism with respect to the DGA-structures induced by the corresponding cup products. In particular, $g^*: H^\bullet (\bar{X}) \to H^\bullet (\tL)$ is a ring homomorphism.
\end{thm}
\begin{proof}
  By revisiting the construction of the spatial homology truncation, we see that for $k<3$ the map $g$ is the inclusion of the base point \cite[Section 1.1.5]{intspaces} and thus the statement of this proposition is the same as the statement of Proposition \ref{inclusionDGA}. So let $k \geq 3$. Then $g= i_\partial \circ f$ with $f$ defined as composition of the subcomplex inclusion
\[
i_{<k}:t_{<k} L\to L/k,
\] 
a homotopy equivalence
\[
h:L/k \to L^k~,
\]
and a further subcomplex inclusion
\[
i_k:L^k\to L
\]
(see \cite[Proposition 1.6\& Section 1.1.6]{intspaces}). We demonstrate in the following how one can choose the cellular diagonal approximations such that each of these maps induces a multiplicative map on cochains.

By Proposition \ref{inclusionDGA}, we can choose a cellular diagonal approximation $\tilde{\Delta}_{L/k}$ on $L/k$ such that ${i_{<k}}^\#$ is a DGA homomorphism with respect to the
cup products $\cup_{\tL} := \tilde{\Delta}_{t_{<k} L}^{\#} \circ \theta_{\tL}$ and $\cup_{L/k} := \tilde{\Delta}_{L/k}^{\#} \circ \theta_{L/k},$ with $\theta_{\tL}$ and $\theta_{L/k}$ the Eilenberg-Zilber maps as in Proposition \ref{inclusionDGA} and $\tilde{\Delta}_{\tL}$ the given diagonal approximation on $\tL$.

Next, we consider the map $h: L/k \to L^k.$ 
The space $L/k$ is a $k$ dimensional CW complex with the same $(k-1)$-skeleton as $L^k$. The $k$ cells are glued in such a way that they correspond to a spacification of a base change in the $k$-th chain group of $L^k$. The map $h$ is constructed as the homotopy inverse of the map
\[
h':L^k\to L/k
\]
relative to the $(k-1)$-skeleton.
The map $h'$ is defined such that it spatially realizes the aforementioned base change. In particular, it induces an isomorphism on the $k$-th chain and cochain group and is the identity on all cochain groups of lower (and trivially also all other) cochain groups. The latter can be formulated as the commutativity of the following diagram.
\[
  \begin{tikzcd}
    L^k \ar{r}{h'} & L/k \\ L^{k-1} \ar{r}{\id} \ar[hook]{u} & L^{k-1} \ar[hook]{u}{i_{k-1}}
  \end{tikzcd}
\]
Note, that it is not obvious that the cellular cochain map induced by $h$ is also an isomorphism a priori. In this setting, it is true, though, as we outline in the following.
In formulas, we know that 
\[
  h' \circ h \simeq id ~ \text{rel}~ L^{k-1}. 
\]
Thus, there is a cochain homotopy operator $s:\C{L/k} \to C^{\cdot-1}(L/k)$ such that
\begin{equation}\label{homotopy}
h^\# \circ {h'}^\# = \id + sd + ds \quad \text{and} \quad i_{k-1}^\# \circ s = 0.
\end{equation}
Since $(i_{k-1}^\#)^r: C^r (L/k) \to C^r(L^{k-1})$ is an isomorphism for $r<k$ and $(i_{k-1}^\#)^{r-1} \circ s^r = 0$ for all $r \in \ZZ,$ we get that $s^r = 0$ for $r\leq k.$ Since $L/k$ is a $k$-dimensional CW-complex, $C^r (L/k) = 0$ for $r>k$ and hence, $s^r = 0$ for all $r \in \ZZ.$ That implies that $h^\# \circ {h'}^\# = \id$. By the same argument, ${h'}^\# \circ h^\# = \id$ also holds and therefore $h^\#: C^\bullet (L^k) \to C^\bullet (L/k)$ is an isomorphism with inverse $(h^\#)^{-1} = {h'}^\#.$

Let $\tilde{\Delta}_{L/k}$ be the cellular approximation used above and set
\[
  \nabla_{L^k} :=(h \times h) \circ \tilde{\Delta}_{L/k} \circ h'.
\]
Since $\nabla_{L^k}$ is cellular, the following calculation shows that $\nabla_{L^k}$ is a cellular approximation of the diagonal $\Delta_{L^k}$ of $L^k$,
\[
\nabla_{L^k}=(h \times h) \circ \tilde{\Delta}_{L/k} \circ h' \simeq (h \times h) \circ \Delta_{L/k} \circ h' = \Delta_{L^k} \circ h \circ h' \simeq \Delta_{L^k}.
\]
Let $\cup_{L^k} = \nabla_{L^k}^\# \circ \theta_{L^k}: C^r (L^k) \otimes C^s(L^k) \to C^{r+s} (L^k)$ be the resulting cup product on $L^k,$ with $\theta_{L^k}$ the Eilenberg-Zilber map as before. By this definition, $h$ becomes a DGA homomorphism since $h^\# \circ {h'}^\# = \id$,
\[ h^\# \circ \cup_{L^k} = h^\# \circ {h'}^\# \circ \tilde{\Delta}_{L/k}^{\#} \circ (h \times h)^\# \circ \theta_{L^k} = \cup_{L/k} \circ (h^\# \otimes h^\#). \]
Finally apply the  homotopy extension and lifting property (cf. \cite[p.~75]{May}) to $ (i_\partial\circ i_k \times i_\partial\circ i_k) \circ \nabla_{L^k}$ and $\Delta_{\bar{X}}$ to obtain a map $\nabla'_{\bar{X}}$ such that $ \nabla'_{\bar{X}}|_{L^k}=\nabla_{L^k}$ and $\nabla'_{\bar{X}} \simeq \Delta_{\bar{X}}.$
Note that $\nabla_{L^k}$ is a composition of cellular maps. So, the cellular approximation theorem relative to $L^k$ yields a cellular map $\nabla_{\bar{X}}$ such that $\nabla_{\bar{X}}|_{L^k}=\nabla_{L^k}$ and $\nabla_{\bar{X}} \simeq \Delta_{\bar{X}}.$
The first equation can be re-written as
\[
\nabla_{\bar{X}}\circ (i_\partial \circ i_k)= \left\{ (i_\partial \circ i_k)\times(i_\partial \circ i_k) \right\} \circ \nabla_{L^k}.
\]
Setting $\cup_{\bar{X}} := \nabla_{\bar{X}}^\# \circ \theta_{\bar{X}},$ where once again $\theta_{\bar{X}}$ is the Eilenberg-Zilber map on $\bar{X}$, $(i_\partial \circ i_k)^\#$ is an DGA homomorphism with respect to the cup products $\cup_{\bar{X}}$ and $\cup_{L^k}$.
\[ (i_\partial \circ i_k)^\# \circ \cup_{\bar{X}} = \cup_{L^k} \circ  \left( (i_\partial \circ i_k)^\# \otimes (i_\partial \circ i_k)^\# \right) \]
We combine the previous results in the following equation.
\[
g^\# \circ \cup_{\bar{X}} = {i_{<k}}^\# \circ h^\# \circ (i_\partial \circ i_k)^\# \circ \cup_{\bar{X}} = \cup_{\tL} \circ (g^\# \otimes g^\#) \]
In summary, $g^\#$ is a DGA homomorphism from $(\C{\bar{X}}, \cup_{\bar{X}})$ to $(\C{\tL}, \cup_{\tL})$ and since those cup products are induced by cellular diagonal approximations, they induce the regular ring structure on cohomology and we arrive at the statement of the proposition.
\end{proof}
\begin{rmk}
Note that after the initial cellular approximation to the diagonal map of $\tL$ we only alter the products on the codomains and always work relative to the subcomplex $\tL$. Thus it is possible to choose products such that $i_0^{\#}$ and $g^{\#}$  simultaneously become DGA homomorphisms. Therefore, $Q^\bullet$ is indeed a pullback in the categories of DGAs.
\end{rmk}
\subsection{$\phi_1^*$ is a ring isomorphism}\label{sect:phi1multqis}
In this section we prove Theorem \ref{isopullback}. We establish the isomorphism from the reduced cohomology ring $\tilde{H}^\bullet(\I)$ to $H^\bullet(Q^\bullet)$ as the following composition of ring isomorphisms
\[
\begin{tikzcd}
  \widetilde{H}^\bullet (\I) \ar{r}{\cong}[swap]{p^*} & H^\bullet (M(g),\tL) \ar{r}{\cong}[swap]{(r|^*)^{-1}}&  H^\bullet (ker(g^{\#})) \ar{r}{\cong}[swap]{\incl^*} & H^\bullet (Q^\bullet)~.
\end{tikzcd}
 \]
Here $p:M(g)\rightarrow \text{cone}(g)=\I$ is the map collapsing $\tL$ embedded as the top of the mapping cylinder $M(g)$ to the cone point $c \in \I$, 
$r:M(g) \to \bar{X}$ is the deformation retraction of the mapping cone onto its base $\bar{X}$
and the last map is the sub-DGA inclusion of $ker(g^{\#})\oplus 0$ in $Q^\bullet$ where both are thought of as sub-DGAs of $\C{\bar{X}}\oplus \rC{\cL}$. Note, that, in the above diagram, $(r|^*)^{-1}$ is an abbreviation for $((r^\#|_{\ker(g^\#)})^*)^{-1}$.
Later we demonstrate that $\phi_1: \rC{\I} \to \Q$ indeed induces the composition $incl^*\circ (r|^*)^{-1} \circ p^*$ on cohomology. We proceed to establish that all the factors in the composition above are ring isomorphisms.

\begin{prop}
The map
\begin{equation}
 p^*: H^\bullet ( \I, \left\{ c \right\}) \xrightarrow{\cong} H^\bullet (M(g), \tL)
  \label{eq:pIsoMappingCylinder}
\end{equation}
is a ring isomorphism.
\end{prop}

\begin{proof}
  By \cite[Lemma 3.1, Chapter X]{MasseySingHomology}, $p^*$ induces an isomorphism of relative cohomology groups in all degrees. Further $p$ is a continuous map and thus induces a multiplicative map on cohomology by the naturality of the cohomological cup product. Hence $p^*$ is a ring isomorphism.
\end{proof}

Recall that under the appropriate choice of products $g^{\#}: C^\bullet (\bar{X}) \to C^\bullet (t_{<k} L)$ is a DGA homomorphism by Theorem \ref{thm:gDGAhom}. Accordingly, $ker(g^{\#})$ is a sub-DGA of $\C{\bar{X}}$. Further note that the retraction $r$ fits into the following commutative diagram of (cellular) maps 
\begin{equation}
  \begin{tikzcd}
    \ & M(g) \ar{dr}{r} & \ \\
    \tL \ar{ur}{i_{\tL}} \ar{rr}{g} & \ & \bar{X},
  \end{tikzcd}
  \label{eq:mappingcylinderg}
\end{equation}
where $i_{\tL}: \tL \hookrightarrow M(g)$ is the inclusion at the top of the cylinder. For $\varphi \in ker(g^{\#})$, the commutativity of Diagram (\ref{eq:mappingcylinderg}) implies that
\[ i_{\tL}^\# \circ r^\# (\varphi) = g^\# \varphi = 0. \]
Accordingly, $r^\#$ restricts to a cochain morphism
\[ r^{\#}|_{ker(g^{\#})}: ker(g^{\#}) \to C^\bullet (M(g), \tL)~. \]
As a deformation retraction, $r$ induces an isomorphism $r^*: H^\bullet (\bar{X}) \xrightarrow{\cong} H^\bullet (M(g))$ of cohomology rings but a priori it is not clear that the restriction of $r^\#$ induces a ring isomorphism 
\begin{equation}
H^\bullet (ker(g^{\#})) \xrightarrow{\cong} H^\bullet (M(g), \tL)~. 
  \label{eq:rIsoMappingCylinder}
\end{equation}
We first establish that $r^{\#}|_{ker(g^{\#})}$ is a quasi-isomorphismus in Proposition \ref{prop:rqis} and then in Proposition \ref{prop:rmult} that the induced map on cohomology is multiplicative. Together, this proves that we indeed have a ring isomorphism on cohomology.

\begin{prop}\label{prop:rqis}
The map
\[ 
r^{\#}|_{ker(g^{\#})}: ker(g^{\#}) \to C^\bullet (M(g), \tL)
\]
is a quasi-isomorphism.
\end{prop}
\begin{proof} 
We apply the 5-Lemma to the pair of long exact sequences on cohomology induced by the following diagram.
    \[
  \begin{tikzcd}
    0 \ar{r} & C^\bullet (M(g), \tL) \ar{r}  		& C^\bullet (M(g)) \ar{r}{i^\#} 		  & C^\bullet (\tL) \ar[equal]{d} \ar{r} & 0 \\
    0 \ar{r} & ker(g^{\#}) \ar{r} \ar{u}{r^{\#}|_{ker(g^{\#})}}	& C^\bullet (\bar{X}) \ar{r}{g^\#} \ar{u}{r^\#} & C^\bullet (\tL) \ar{r} & 0
  \end{tikzcd}
\]
The top row is the usual short exact sequence of relative cochains and the bottom row is exact since $g^{\#}$ is surjective as we establish below. The diagram is obviously commutative and since the last two vertical maps are either an isomorphism or quasi-isomorphism, the map induced by $r^{\#}|_{ker(g^{\#})}$ is an isomorphism, too.

Let us now establish that $g^{\#}$ is surjective. Since the restriction $g|: (\tL)^{k-1} \to \bar{X}$ to the $(k-1)$-skeleton is a subcomplex inclusion and $\tL$ has no cells of dimension greater than $k$, we only have to prove, that $g^{\#,k}: C^k (\bar{X}) \to C^k (\tL)$ is surjective. The map factors as follows:
  \[ C^k (\bar{X}) \xrightarrow{i_{\partial}^{\#,k}} C^k (L) \xrightarrow{i_k^{\#,k}} C^k (L^{k}) \xrightarrow{h^{\#,k}} C^k (L/k) \xrightarrow{i_{<k}^{\#,k}} C^k (\tL). \]
 Recall from the proof of Theorem \ref{thm:gDGAhom} that $h^{\#,k}$ is an isomorphism and all the other maps are induced by CW-subcomplex inclusions and henceforth induce surjective cochain maps. Hence, $g^\#$ is surjective.
\end{proof}

Now we are left with the proof that $r^{\#}|_{ker(g^{\#})}$ induces a multiplicative map on cohomology.

\begin{prop}\label{prop:rmult}
The map
\[ 
r^{\#}|_{ker(g^{\#})}: ker(g^{\#}) \to C^\bullet (M(g), \tL)
\]
induces a multiplicative map on cohomology.
\end{prop}

\begin{proof}
We know that $r$ maps the $(k-1)$-skeleton $L^{k-1}$ of the top of the cylinder $M(g)$ to $L^{k-1} \subset L \subset \bar{X}$, since the restriction of $g| : L^{k-1} \hookrightarrow \bar{X}$ is a CW-subcomplex inclusion. We fix a cellular approximation $\nabla_{L^{k-1}}$ to the diagonal $\Delta_{L^{k-1}}$ of $L^{k-1}$ and then approximate the the diagonals $\Delta_{\bar{X}}: (\bar{X},L^{k-1}) \to (\bar{X},L^{k-1}) \times (\bar{X},L^{k-1})$ and $\Delta_{M(g)}: (M(g), L^{k-1}) \to (M(g), L^{k-1}) \times (M(g), L^{k-1})$ relative to $\nabla_{L^{k-1}}$ to obtain cellular maps $D_{\bar{X}}\simeq \Delta_{\bar{X}}$ and $D_{M(g)} \simeq \Delta_{M(g)}.$ Since $\Delta_{\bar{X}} \circ r = (r \times r) \circ \Delta_{M(g)}$, also the following maps are homotopic rel $L^{k-1}$.
\[
D_{\bar{X}} \circ r \simeq \Delta_{\bar{X}} \circ r = (r \times r) \circ \Delta_{M(g)} \simeq (r \times r) \circ D_{M(g)} ~ \text{rel}~ L^{k-1}. \]
In detail, there is a cochain homotopy $s: C^{\bullet} (\bar{X} \times \bar{X}) \to C^{\bullet -1} (M(g))$ such that for the inclusions $j_{k-1}: L^{k-1} \hookrightarrow L \hookrightarrow \bar{X}$ and $i_{k-1}: L^{k-1} \hookrightarrow L \hookrightarrow M(g)$, and $s_{L^{k-1}}: C^{\bullet} (L^{k-1} \times L^{k-1}) \to C^{\bullet -1} (L^{k-1})$ the cochain homotopy induced by $\nabla_{L^{k-1}}\simeq \Delta_{L^{k-1}}$ the following holds.
\begin{align*} 
  i_{k-1}^\# \circ s &= s_{L^{k-1}} \circ (j_{k-1} \times j_{k-1})^\# \quad \text{and} \\
  r^\# \circ D_{\bar{X}}^\# &= D_{M(g)}^\# \circ (r \times r)^\# + d_{M(g)} \,\circ \, s + s \,\circ\, d_{\bar{X} \times \bar{X}}.
\end{align*}
Let $\varphi \in ker(g^{\#})^l$ and $ \psi \in ker(g^{\#})^m$ be closed. Then the second of the above relations implies that
\begin{align*}
  r^\# \circ D_{\bar{X}}^\# (\varphi \times \psi) 
  &= D_{M(g)}^\# \circ (r \times r)^\# (\varphi \times \psi) + d_{M(g)} \, s (\varphi \times \psi) + s \, d_{\bar{X} \times \bar{X}} (\varphi \times \psi) \\
  &= D_{M(g)}^\# \circ (r \times r)^\# (\varphi \times \psi) + d_{M(g)} \, s (\varphi \times \psi).
\end{align*}
The last summand in the first line vanishes since $\varphi \times \psi$ is closed as the cross product of two closed forms. In the following, we show that there is a cochain $\alpha \in C^{l+m-1} (M(g), \tL)$ with $ d \, s (\varphi \times \psi) = d \alpha.$ This statement is equivalent to the multiplicativity of the map $r|^*: H^\bullet (ker(g^{\#})) \to H^\bullet (M(g), \tL)$.
We distinguish the cases $l+m \leq k,$ $l+m = k+1$ and $l+m > k+1.$ 

First, let $l+m > k+1$. Then $s \, (\varphi \times \psi)$ is a cochain of degree greater than $k$. Since $\tL$ has no cells of dimension greater than $k$, it follows automatically that the pullback under the inclusion $i_{\tL}: \tL \hookrightarrow M(g)$ of $s (\varphi \times \psi)$ is zero. Therefore, $s (\varphi \times \psi) \in C^{l+m-1} (M(g), \tL)$ and we might choose $\alpha:=  s (\varphi \times \psi)$. 
Now, let $l+m \leq k.$ The inclusion $i_{k-1}': L^{k-1} \hookrightarrow \tL$ of the $(k-1)$-skeleton fits in the following commutative diagram, where, $i_{\tL}: \tL \hookrightarrow M(g)$ is the inclusion of the top of the cylinder as in (\ref{eq:mappingcylinderg}) and $j_{k-1}:L^{k-1}\hookrightarrow \bar{X}$ is the inclusion of $L^{k-1}$ as part of the boundary of $\bar{X}$.
\[
  \begin{tikzcd}
    \ & \bar{X} & \ \\
    L^{k-1} \ar[hook]{rr}{i_{k-1}'} \ar[hook]{rd}[swap]{i_{k-1}} \ar{ru}[hook]{j_{k-1}}& \ & \tL \ar[hook]{dl}{i_{\tL}} \ar{ul}[swap]{g}  \\ 
    \ & M(g) & \   \ 
  \end{tikzcd}
\]
The induced map $i_{k-1}'^\#: C^p (\tL) \xrightarrow{\cong} C^{p} (L^{k-1})$ is an isomorphism on cochains for all $p \neq k$. Therefore, we get the following relation.
\begin{align*}
  i_{\tL}^\# s (\varphi \times \psi) &= (i_{k-1}'^\#)^{-1} i_{k-1}^\# s (\varphi \times \psi) = (i_{k-1}'^\#)^{-1} s_{L^{k-1}} (j_{k-1} \times j_{k-1})^{\#} (\varphi \times \psi) \\
  &= (i_{k-1}'^\#)^{-1} s_{L^{k-1}} (i_{k-1}' \times i_{k-1}')^{\#} (g \times g)^\# (\varphi \times \psi) = 0. 
\end{align*}
The last equality holds, since $(g \times g)^\# (\varphi \times \psi) = (g^\# \varphi) \times g^\# \psi = 0 \times 0 = 0.$ 
As in the previous case, we can thus set $\alpha := s (\varphi \times \psi) \in C^{l+m-1} (M(g), \tL)$.

Last, we consider the case $l+m = k+1$. Recall that $\Hk{\C{\tL}}=0$ and $ker(d^k)=\Ck{\tL}$ so the differential $d_{\tL}^{k-1}: C^{k-1}(\tL) \to C^{k}(\tL)$ is surjective.
 Accordingly,
 $i_{\tL}^\# s (\varphi \times \psi) = d \beta,$ for some $\beta \in C^{k-1} (\tL)$. Since the pullback $i_{\tL}^\#: C^{k-1} (M(g)) \to C^{k-1} (\tL)$ under the CW-subcomplex inclusion $i_{\tL}$ is surjective too, there is a cochain $\gamma \in C^{k-1} (M(g))$ with $\beta = i_{\tL}^\# \gamma$ and we get the following equation.
\[ i_{\tL}^\# s (\varphi \times \psi) = d \beta = d i_{\tL}^\# \gamma = i_{\tL}^\# (d \gamma). \]
This is equivalent to stating that $s(\varphi \times \psi) - d \gamma \in C^{k} (M(g), \tL)$. We can then set $\alpha = s(\varphi \times \psi) - d \gamma,$ since $d \left( s (\varphi \times \psi) - d \gamma \right)  = d \left( s (\varphi \times \psi) \right).$ 
\end{proof}

Proposition \ref{prop:rqis} and \ref{prop:rmult} together establish that $r|^*$ is a ring isomorphism.  

\begin{lem}
  The map
  \begin{align*}
   \incl:  ker(g^{\#}) \hookrightarrow Q^\bullet, \quad
  \varphi  \mapsto (\varphi, 0)
\end{align*}
induces a ring isomorphism on cohomology.
\end{lem}
\begin{proof}
Since $Q^{\bullet}$ was defined as the pullback in diagram (\ref{pullbackQ}), it can be explicitely written as 
\[
Q^\bullet = \left\{ (\varphi, \psi) \in C^\bullet (\bar{X}) \oplus \rC{\cL} | g^\# \varphi = i_0^{\tilde{\#}} \psi \right\}
\]
 The boundary map and cup product on $Q^\bullet$ are restrictions of the direct sums of the boundary maps and cup products in $C^\bullet (\bar{X})$ and $\rC{\cL}$. As established before $ker(g^{\#})$ is a sub-DGA of $C^\bullet (\bar{X})$. Therefore $incl$ is a sub-DGA inclusion and induces a multiplicative map on cohomology.

To see that $incl$ is a quasi-isomorphism consider the following diagram.
\[
  \begin{tikzcd}[column sep=small]
    0 \ar{r} & Q^\bullet \ar{r} & C^\bullet (\bar{X}) \oplus \rC{\cL} \ar{r}{F} & C^\bullet (\tL) \ar{r} & 0 \\
    0 \ar{r} & ker(g^{\#}) \ar[hook]{u}{incl} \ar{r} & C^\bullet (\bar{X}) \ar[hook]{u}{qis} \ar{r}{g^\#} & C^\bullet (\tL) \ar[equal]{u} \ar{r} & 0
  \end{tikzcd}
\]
The map $F$ sends $(\varphi, \psi) \mapsto g^\# \varphi - i_0^\# \psi$ and is surjective, since $g^\#$ is surjective as was established in the proof of Proposition \ref{prop:rqis}. The lower sequence also appeared already in this proof. The vertical map in the middle is the inclusion as the first factor. It is a quasi-isomorphism, since all the reduced cohomology groups of the cone of $\tL$ vanish. Applying the 5-Lemma to the induced pair of long exact cohomology sequences then proves that $incl$ is a quasi-isomorphism.
\end{proof}

Combining all the results of this section proves Theorem \ref{isopullback}. For later purposes, we establish the following concrete description of that isomorphism. 

\begin{thm}
The map $\phi_1$ induced by the universal property of the pullback in diagram 
\[
\begin{tikzcd}
\rC{\I} \drar{\phi_1} \arrow{ddr}[swap]{i_{\bar{X}}^{\tilde{\#}}} \arrow{rrd}{i_{\cL}^{\#}}                        &&\\
& Q^{\cdot} \arrow{r} {q_2}\arrow{d}{q_1}& \rC{\cL} \arrow{d}{i_0^{\tilde{\#}}} \\
&\C{\bar{X}} \arrow{r}{g^{\#}}& \C{\tL}
\end{tikzcd}
\]
 induces the ring isomorphism given by the composition
 \[
\begin{tikzcd}
  \widetilde{H}^\bullet (\I) \ar{r}{\cong}[swap]{p^*} & H^\bullet (M(g),\tL) \ar{r}{\cong}[swap]{(r|^*)^{-1}} & H^\bullet (ker(g^{\#})) \ar{r}{\cong}[swap]{\incl^*} & H^\bullet (Q^\bullet),
\end{tikzcd}
 \]
 where $r|^*$ is an abbreviation of $(( r^{\#}|_{ker(g^{\#})})^*)^{-1}$.
 \end{thm}

\begin{proof}

We equip $\cL$ with the canonical cell structure induced from the cell structure of $\tL$ (see \cite[Section 2.3]{Fritsch} for more details) and choose the apex of the cone as base point. In this situation
\[
\rC{\cL}=\C{\tL}\oplus C^{\cdot -1}(\tL),
\]
with differential 
\[{\left(\begin{smallmatrix} d_{\tL} & 0 \\ id_{\tL}^{\#} & -d_{\tL}[-1] \end{smallmatrix} \right).}\]
Then, for a cochain $(b_1,b_2) \in \rC{\cL}$, we get $i_0^{\#}(b_1, b_2)=b_1.$
Analogously to $\cL$, we equip $\I= cone(g)$ with the canonical cell structure induced from the cell structures of $\tL$ and $\bar{X}$ and again choose the apex of the cone as the base point. Thus,
\[
\rC{\I}=\C{\bar{X}}\oplus C^{\cdot -1}(\tL),
\]
with differential 
\[{\left(\begin{smallmatrix} d_{\bar{X}} & 0 \\ g^{\#} & -d_{\tL}[-1] \end{smallmatrix} \right)}.\]
Due to the canonical choices of the cell structures of the cones, we have 
\[
i_{\cL}^{\#}(a,b)=(g^{\#}a,b)
\]
and 
\[
i_{\bar{X}}^{\tilde{\#}}(a,b)=a~.
\]
Since
\[
i_0^{\tilde{\#}}\circ i_{\cL}^{\#} (a,b)=i_0^{\tilde{\#}}(g^{\#} a,b)=g^{\#} a= g^{\#}\circ i_{\bar{X}}^{\tilde{\#}} (a,b)~,
\]
the outer square in the diagram above commutes and the universal property defines the map $\phi_1$ uniquely.
As mentioned at the beginning of section \ref{sect:pullback}, $\phi_1$ is a cochain map by construction but it is not clear whether it is a DGA homomorphism or whether it at least induces a multiplicative map on cohomology. We prove now that $\phi_1$ induces the same map on cohomology as the composition that we have constructed previously.

Considering $Q^\bullet$ as a subset of $\C{\bar{X}}\oplus \rC{\cL}$ allows us to write
\[
\phi_1:\C{\bar{X}}\oplus C^{\cdot -1}(\tL) \to Q^{\cdot}, (a,b)\mapsto (a,(g^{\#}a,b))
\]
explicitly.
If $(a,b)$ is a closed cochain, it holds that $da = 0$ and $g^\# a = db$. Since $g^\#: \C{\bar{X}} \to \C{\tL}$ is surjective, we can choose a cochain $\beta \in \C{\bar{X}}$ with $g^\# \beta = b.$ Then, $(a - d \beta, (0,0) ) \in \Q$ is a representantive of the cohomology class of $\phi_1(a,b)$, which can be seen by the following calculation.
    \[
      \left( a - d \beta, \left( 0,0 \right) \right) + d \left( \beta, (b,0) \right) = \left( a - d \beta, (0,0) \right) + (d\beta, (db, b)) = \left( a, (db, b) \right) = \phi_1 (a,b).
    \]
    
Since $g^\# (a - d \beta) = g^\# a - d b = 0,$ we can write $\left( a - d\beta, (0,0) \right)= \incl(a - d\beta)$. In analogy to the models used for cellular cochain complexes of $\cL$ and $\I$, the cellular cochain complex of the mapping cylinder $M(g)$ can be identified with $\C{M(g)} = \C{\bar{X}} \oplus C^{\bullet -1} (\tL) \oplus \C{\tL}$ with differential
    \[ \left( \begin{smallmatrix} d_{\bar{X}} & 0 & 0 \\ g^\# & - d_{\tL} & \id \\ 0 & 0 & d_{\tL} \end{smallmatrix} \right). \]
    The complex relative to the top of the cylinder becomes $\C{M(g),\tL} = \C{\bar{X}} \oplus C^{\bullet -1 } (\tL)$ with induced differential.
    The cochain map $r^\#: \C{\bar{X}} \to \C{M(g)}$ maps $a \in \C{\bar{X}}$ to $(a, 0 , g^\# a)$. 
    Therefore, the restriction $r^\# : ker(g^{\#}) \to \C{M(g),\tL}$ maps $a \in \C{\bar{X},\tL}$ to $(a,0) \in \C{M(g),\tL}.$ The cochain map $p^\#: \rC{\I} \to \C{M(g),\tL}$ becomes $p^\#(a,b) = (a,b)$. The following calculation then shows that $p^* \left( [(a,b)] \right) = r|^* \left( [a - d\beta] \right)$ and thus implies that $\phi_1^*$ is the composition of the proposition, 
    \[ p^\# (a,b) - r^\# (a - d\beta) = (a,b) - (a - d \beta, 0) = (d \beta, b ) = (d \beta ,g^\# \beta) = d ( \beta, 0)~. \]
    \end{proof}

\section{A Multiplicative de Rham Theorem for HI}\label{section_multdR}

The goal of this section is to prove that the cohomology rings $\HH{\WI}$ and  $\HH{\C{\I}}$ are isomorphic. In Section \ref{cellulardR}, we extend the de Rham map to map from singular cochains to cellular cochains since we are forced to use cellular cochains in Section \ref{sect:middlepart}. This is established for both the absolute and relative case and is then applied in Section \ref{sect:deRhamPart} to construct the first part $\tilde{\rho}$ of our eventual intersection space cohomology de Rham map $\phi$. For each case we demonstrate that the maps induce multiplicative maps on cohomology. In Section \ref{sect:pullback} we established that the cellular cochains of $\I$ fit up to the isomorphism $\phi_1$ into a pullback square. This property is now used  to construct a map $\phi_2$ in Section \ref{sect:middlepart} that combines with $\tilde{\rho}$ and $\phi_1^{-1}$ into $\phi$. We take care to construct $\phi_2$ as DGA homomorphisms. Accordingly the induced map on cohomology is multiplicative and so is the induced map of $\phi$. Section \ref{sect:formphi} gives the explicit form of $\phi$ on cochain level and Section \ref{qisphi} establishes that $\phi$ is a quasi-isomorphism. Accordingly the induced map is a ring isomorphism and we have proven our result.

\subsection{de Rham Map to Cellular Cochains}\label{cellulardR}
Let $C_\bullet (M)$ denote the cellular, $S_\bullet (M)$ the singular and $\widehat{S}_\bullet (M) $ the normalized singular chain complex of $M$, $C^\bullet (M),$ $S^\bullet (M)$ and $ \widehat{S}^\bullet (M)$ the corresponding cochains complexes and define relative chains and cochains as usual. 

In the rest of this article, we work with a de Rham type map $\rho: \Omega^\bullet (M) \rightarrow C^\bullet (M)$ and the corresponding relative morphism. $M$ is a smooth manifold, possibly open or with boundary, with some CW decomposition. In the relative setting, $L$ is a submanifold of $M$ and we choose CW decompositions such that $L \subset M$ is a CW-subcomplex.  We outline the construction of the cellular de Rham maps $\rho$ and $\rho_{rel}$ and explain why they induce multiplicative isomorphisms on cohomology. 

We can restrict a singular cochain to the nondegenerate simplices and get an element of $\widehat{S}^n (M).$ This defines a restriction operator
\[ \text{restr}: S^n (M) \twoheadrightarrow \widehat{S}^n (M), \]
which is clearly multiplicative.
The geometric realization $\Gamma M$ of the non-degenerate singular simplices of $M$ is a CW complex with one $n-$cell for each non-degenerate singular $n-$simplex and with cellular chain complex $C_\bullet (\Gamma M)$ naturally isomorphic to $ \widehat{S}_\bullet (M).$ Further, there is a weak equivalence $\gamma: \Gamma M \rightarrow M$. In our setting the Whitehead Theorem implies that $\gamma$ actually is a homotopy equivalence. Taking a CW-approximation of the homotopy inverse of $\gamma$, which is still a homotopy equivalence and which we denote by $\delta: M \rightarrow \Gamma M$, this gives a cochain homotopy equivalence 
\( \delta^{\#} : \widehat{S}^\bullet (M) = C^\bullet (\Gamma M) \rightarrow C^\bullet (M) \) 
that induces a ring isomorphism on cohomology. This construction also carries over to the relative case.

The last component of the multiplicative de Rham isomorphism we use henceforth is Lee's smoothing operator (see \cite[pp. 474 ff]{Lee}). It is a chain homotopy equivalence $s: S_\bullet (M) \rightarrow S_\bullet^\infty (M) $, where $S_\bullet^\infty (M) $ is the chain complex of smooth singular chains. Given a submanifold $L \subset M,$ the operator can be defined in such a way that it commutes with the inclusion of this submanifold in $M$.
Hence, $s$ induces a quasi-isomorphism of relative chain complexes 
\[ s: S_\bullet (M,L) \xrightarrow{qis} S_\bullet^\infty (M,L).\] 
The induced maps on the absolute and relative cochain complexes are also quasi-isomorphisms and therefore induce isomorphisms, i.e.
\[ 
\begin{split}
s^*: ~& H_{sing, \infty}^\bullet (M) \xrightarrow{\cong} H_{sing}^\bullet (M), \\
s^*: ~& H_{sing, \infty}^\bullet (M,L) \xrightarrow{\cong} H_{sing}^\bullet (M,L).
\end{split} 
\]
Since the inverses are induced by the restrictions $ S^\bullet (M) \rightarrow S_\infty^\bullet (M) $ and $ S^\bullet (M,L) \rightarrow S_\infty^\bullet (M,L) $, which are clearly multiplicative, the induced maps $s^*$ on cohomology are ring isomorphisms.

\vspace{0.2cm}
\noindent
Let us denote the linear dual to $s$ by $s^\dagger$. We define the de Rham maps 
$ \rho: \Omega^\bullet (M) \rightarrow C^\bullet (M) $ 
and $\rho_{rel}: \Omega^\bullet (M,L) \rightarrow C^\bullet (M,L) $ 
by the following commutative diagrams. 
\[ 
\begin{tikzcd}
\Omega^\bullet (M) \ar{r}{\rho_s} \ar{ddrr}[swap]{\rho} 
& S_\infty^\bullet (M) \ar{r}{s^\dagger} & S^\bullet (M) \ar{d}{restr} \\
\ & \ & \widehat{S}^\bullet (M) \ar{d}{\delta^{\#}} \\
\ & \ & C^\bullet (M),
\end{tikzcd}
\]
where $\rho_s$ is the classical de Rham map and 
\[ 
\begin{tikzcd}
\Omega^\bullet (M, L) \ar{r}{\rho_{s,rel}} \ar{ddrr}[swap]{\rho_{rel}} 
& S_\infty^\bullet (M,L) \ar{r}{s^\dagger} & S^\bullet (M,L) \ar{d}{restr} \\
\ & \ & \widehat{S}^\bullet (M,L) \ar{d}{\delta_{rel}^{\#}} \\
\ & \ & C^\bullet (M,L),
\end{tikzcd}
\]
where $\rho_{s,rel} $ is the relative de Rham map of Theorem \ref{thm:rel_dR_mult}. All the maps we used to define the de Rham maps $\rho$ and $\rho_{rel}$ are DGA morphisms, or at least induce ring isomorphisms on cohomology. For the classical de Rham maps this follows by \cite[Theorem III-3.1]{Bredon_ST} in the absolute case and by the relative de Rham Theorem in Section \ref{subs:rel_derham_map} in the relative case. Hence the maps $ \rho$ and $\rho_{rel}$ induce ring isomorphisms on cohomology. This proves the following theorem.

\begin{restatable}[Cellular multiplicative relative de Rham Theorem]{thm}{cmrdRT}\label{hyp}
Let $L\subset M$ be a smooth submanifold of the smooth manifold $M$. Choose a CW-structure on M such that $L \subset M$ is also a CW-subcomplex.
Then, the map $\rho_{rel}$ induces a ring isomorphism 
\[
\rho_{rel}^*: H_{dR}^{\bullet}(M,L) \to H^{\bullet} \left(C^{\bullet}(M,L) \right).
\]
\end{restatable}
 
 \subsection{Multiplicativity of $\tilde{\rho}$}\label{sect:deRhamPart}

In this section, we construct a map $\tilde{\rho}$ that translates the construction of $\WI$ to cochains of CW complexes. We then use Theorem \ref{hyp} to prove that $\tilde{\rho}$ induces a multiplicative map on cohomology.

Recall that
\[
 \WI:=\{\omega\in \WX| i_\partial^{*}(\omega)\in \TW\}~,
 \]
so the de Rham map $\rho_{\bar{X}}$ on $\WX$ restricts to a map $\rho_{\bar{X}}|$ on $\WI$. By the naturality of the de Rham map its restriction $\rho_{\bar{X}}|$ factors over
\[
\Rb:=\{x \in \C{\bar{X}}| i_\partial^{\#}(x)\in \L \}~
\]
with $\L$ the naive cotruncation (i.e. $0$ in degrees lower than $k$ and $\C{L}$ in degrees greater than or equal to $k$). We define $\tilde{\rho}$ to be this factor. So with  $incl.:\Rb\to \C{\bar{X}}$ the sub-complex inclusion we have $\rho_{\bar{X}}|=incl.\circ \tilde{\rho}$.


Note that with the definition
\[
\Crel :=\{x \in \C{\bar{X}}| i_\partial^{\#}(x)=0\}
\]
we have sub-complex inclusions $\Crel \subset \Rb \subset \C{\bar{X}}$, analogously to the inclusions $\Omega^\bullet(\bar{X},L)\subset \WI \subset \Omega^\bullet(\bar{X})$. Recall that in the proof of Theorem \ref{thm:gDGAhom} we established that $i_{\partial}^{\#}$ is multiplicative and the product $\cup_{\bar{X}}$ on $\C{\bar{X}}$ from Section \ref{sect:pullback} restricts to $\Crel$. The restriction of $\cup_{\bar{X}}$ to $\Rb$ is also well-defined due to the construction of $\Rb$ via cotruncation together with the graded nature of the cup product. Therefore the inclusions above are sub-DGA inclusions. 

\begin{prop}\label{rhomult}
The map $\tilde{\rho}$ induces a multiplicative map on cohomology.
\end{prop}
\begin{proof}
Let $\omega \in \Wbp$ and $\eta \in \Wbq$ be two closed forms. Let us consider the case $q=0$ and $r$ arbitrary. The case $q$ arbitrary and $r=0$ is analogous. For $q=0$, the closed form $\omega$ is a constant function. Recall that $k=n-1-\bar{p}(n)$ with $\bar{p}$ a Goresky-MacPherson perversity function. The definition of these perversity functions directly implies $k\geq 1$. Therefore, $\WIzero=\Omega^0(\bar{X},L)$ and $\omega$ has to vanish on the boundary. We conclude that $\omega=0$ and 
\begin{align*}
\tilde{\rho}(\omega) \cup \tilde{\rho} (\eta)&=0 \cup \tilde{\rho} (\eta)\\
                            &=0\\
                            &=\tilde{\rho}(0 \wedge\eta)\\
                            &=\tilde{\rho}(\omega \wedge\eta).
\end{align*}
Thus the multiplicativity already holds on cochain level.

Next, consider the case $q$ and $r<k$. Here, we have $\omega \in \Wrelp$ and $\eta \in \Wrelq$. Using that $\rho_{rel}$ is a restriction of $\tilde{\rho}$ we calculate

\begin{align*}
\tilde{\rho}(\omega) \cup \tilde{\rho} (\eta)&=\rho_{rel}(\omega) \cup \rho_{rel} (\eta)\\
                                &=\rho_{rel}(\omega \wedge\eta)+ dx
\end{align*}
for some $x \in \Crelpqq$ since $\rho_{rel}$ induces a multiplicative map on cohomology if we take the relative de Rham theorem into account. Since $\Crelpqq \subset \Rbpqq$, we have $x\in \Rbpqq$. Further, we observe that since $\omega$ and $\eta$ vanish on the boundary, their wedge product does so, too. So,
$ \rho_{rel}(\omega \wedge\eta)=\tilde{\rho}(\omega \wedge\eta)$ 
and, combing all these facts, we see 
\begin{align*}
[\tilde{\rho}(\omega) \cup \tilde{\rho} (\eta)]=[\tilde{\rho}(\omega \wedge\eta)+dx]
=[\tilde{\rho}(\omega \wedge\eta)]
\end{align*} 
with $[\dots]$ denoting the cohomology class in the cohomology of $\Rb$.

Finally, let $q\geq k$ and $r\geq 1$, and in particular $q+r\geq k+1$. We use $\WI\subset \WX$ and that $\tilde{\rho}$ is a restriction of $\rho_{\bar{X}}$ to calculate
\begin{align*}
\tilde{\rho}(\omega) \cup \tilde{\rho} (\eta)&=\rho_{\bar{X}}(\omega) \cup \rho_{\bar{X}} (\eta)\\
&=\rho_{\bar{X}}(\omega \wedge\eta)+ dx = \tilde{\rho} (\omega \wedge \eta) + dx
                                \end{align*}                               
for some $x \in C^{q+r-1}(\bar{X}) $. The last line follows because the classical de Rham map is mutliplicative on cohomology. We used that if $q+r\geq k+1$, then  $\tilde{\rho}=\rho_{\bar{X}}$. Further, $C^{q+r-1}(\bar{X}) = \Rbpqq$ and the multiplicativity also holds in the cohomology of $\Rb$, analogously as for the case above.
\end{proof}

\subsection{Multiplicativity of $\phi_2$}\label{sect:middlepart}

In Section \ref{sect:pullback} and \ref{sect:deRhamPart} we obtained the maps $\phi_1 : \C{\I} \to Q^\bullet$ and $\tilde{\rho} :\WI \to \Rb$ which induce ring isomorphisms on cohomology. To complete our intersection space cohomology de Rham map, we construct the connecting piece $\phi_2 : \Rb \to Q^\bullet$ via the universal property of the pullback $Q^\bullet$ applied to the diagram
\begin{equation}\label{diag:pullbackphi2}
\begin{tikzcd}
\Rb \arrow{ddr}[swap]{incl}\arrow{r}{i_\partial^{\#}}  \arrow{dr}{\phi_2}& \L \arrow{dr}{\mathfrak{X}} &  \\
& Q^{\cdot} \arrow{r}{q_2} \arrow{d}{q_1} & \rC{\cL} \arrow{d}{i_0^{\tilde{\#}}} \\
& \C{\bar{X}} \arrow{r}{g^{\#}}& \C{\tL}~.
\end{tikzcd}
\end{equation}
We construct the map $\mathfrak{X}: \L \rightarrow \rC{\cL}$ in Lemma \ref{gamma} and establish the commutativity of the diagram in Lemma \ref{outerperimeter}. We also establish that Diagram \ref{diag:pullbackphi2} is a diagram in the category of DGAs. Therefore $\phi_2$  a DGA homomorphism.

\begin{lem}\label{gamma}
There is a DGA homomorphism $\mathfrak{X} :\L \to \rC{\cL}$ that satisfies $i_0^{\tilde{\#}}\circ \mathfrak{X} = f^{\#}|_{\L}$. Recall that we always assume $k\geq 1$.
\end{lem}
\begin{proof}
Note that the naive cotruncation $\L$ vanishes in degrees smaller than $k$ and $\rC{\cL}$ vanishes in degrees greater than $k+1$. Accordingly, $\mathfrak{X}$ is the trivial map in degrees different from $k$ or $k+1$.

To obtain a cochain map we need to choose $\mathfrak{X}^k$ and $\mathfrak{X}^{k+1}$ such that the diagram
\[
\begin{tikzcd}         
   \Ck{L}    \rar{\mathfrak{X}^k}\dar{d_{L}}&\rCk{\cL}\dar{d_{cone}}\\            
   C^{k+1}(L)    \rar{\mathfrak{X}^{k+1}} & \widetilde{C}^{k+1}(\cL)                 
\end{tikzcd}
\]
commutes.
Using the canonical choice for a cell structure of $\cL$ and choosing the cone point as base point, we again identify 
\[
\rCk{\cL}=C^k(\tL)\oplus C^{k-1}(\tL)
\]
and 
\[ \widetilde{C}^{k+1}(\cL)= 0 \oplus C^k(\tL).
\]
Recall that with this identification, the differential becomes
\[
d_{cone}:={\left(\begin{smallmatrix} d_{\tL} & 0\\ id_{\C{\tL}} & -d_{\tL}[-1]   \end{smallmatrix} \right)}.
\]

Recall that $\Hk{\C{\tL}}=0$ and $ker(d^k)=\Ck{\tL}$ so the differential $d_{\tL}^{k-1}: C^{k-1}(\tL) \to C^{k}(\tL)$ is surjective.
 Accordingly, we might choose a linear section $\mathfrak{x}: C^{k}(\tL) \to C^{k-1}(\tL)$ of $d_{\tL}^{k-1}$, i.e.
\[
d_{\tL}^{k-1} \circ \mathfrak{x} = id~.
\]

We define 
\[
\mathfrak{X}^k :=(id_{C^k(\tL)}\oplus \mathfrak{x})\circ f^{\# ,k}|_{\L}~,
\]
and
\[
\mathfrak{X}^{k+1} :=0. 
\]
Here $f$ is the map from the construction of the Moore approximation and $f^{\# ,k}$ denotes the part in degree $k$ of the induced cochain map $f^{\#}$.  The computation 
\begin{align*}
 d_{cone}^{k} \circ \mathfrak{X}^{k} (x)	  
 = & 
\begin{pmatrix}
d_{\tL}^{k} \circ f^{\# ,k}|_{\L} (x) \\ f^{\# ,k}|_{\L}(x) -d_{\tL}^{k-1} \circ \mathfrak{x}  \circ f^{\# ,k}|_{\L}(x) 
\end{pmatrix}
 \\
&= 
\begin{pmatrix}
0\\ f^{\# ,k}|_{\L}(x) -f^{\# ,k}|_{\L}(x) 
\end{pmatrix}
 \\
&=0 \; = \mathfrak{X}^{k+1} \circ d^{k}_L(x)
\end{align*} 
with $x$ an element of $\Ck{L}$ proves the commutativity of the square. Thus $\mathfrak{X}$ is a cochain map (we made use of $d^{k}_{\tL}= 0$ and the definition of $\mathfrak{x}$ as a section of $d_{\tL}^{k-1}$).

The multiplicativity of $\mathfrak{X}$ is equivalent to the commutativity of the diagram
\[
\begin{tikzcd}
\Lp \otimes \Lq \arrow{rr}{\cup|} \arrow{d}{\mathfrak{X}\otimes \mathfrak{X}} & & \Lpq \arrow{d}{\mathfrak{X}} \\
\rCp{\cL} \otimes \rCq{\cL} \arrow{rr}{\cup} && \rCpq{\cL}
\end{tikzcd}
\]
for all combinations of degrees $q$ and $r$. Here $\cup|$ is the restriction of the cup product of $\C{L}$ to $\L$ and $\cup$ the cup product on $\rC{\cL}$.
If either $q<k$ or $r<k$ then $\Lp \otimes \Lq=0$ and the diagram commutes trivially. On the other hand, if $q\geq k$, $r\geq k$ and $k>1$ we can infer that 
\[q+r>k+1\geq dim(\cL).\]
Thus $\rCpq{\cL}=0$ and the commutativity of the diagram is given for trivial reasons. Recall that $L$ is assumed to be simply connected. In this situation $t_{<1}L:=\{pt\}$ and $f$ the inclusion of the base point constitute a spatial homology truncation of $L$ with $k=1$. Accordingly $cone (t_{<1}L)$ is actually a one dimensional CW complex. On the other hand $k=1$ together with $q\geq k$ and $r\geq k$ implies $q+r\geq 2$ and thus we get the same situation as in the case with $k>1$. In conclusion, the diagram commutes for all degrees $q$ and $r$ independent of the cut-off value $k$. Thus $\mathfrak{X}$ is not only a cochain map but a DGA homomorphism for all $k$.

Finally, recall that $i_0^{\tilde{\#}}(b_1,b_2)=b_1$, implying $i_0^{\tilde{\#}} \circ \mathfrak{X} = f^{\#}|_{\L}$.
\end{proof}

Let us point out that the proof above and especially the part concerning the multiplicativity of $\mathfrak{X}$ abused the fact that the cellular cochain complexes vanish above the dimension of the space. This is the primary motivation to work with cellular cochains in this article. In order to justify the construction of $\phi_2$ via the universal property of the pullback we prove

\begin{lem}\label{outerperimeter}
The diagram
\[
\begin{tikzcd}
\Rb \arrow{dd}[swap]{incl}\arrow{r}{i_\partial^{\#}}  & \L \arrow{d}{\mathfrak{X}}   \\
 & \rC{\cL} \arrow{d}{i_0^{\tilde{\#}}} \\
 \C{\bar{X}} \arrow{r}{g^{\#}}& \C{\tL}~.
\end{tikzcd}
\]
is commutative and products can be chosen such that all maps are DGA homomorphisms.
\end{lem}
\begin{proof}Direct computation yields
\begin{align*}
i_0^{\tilde{\#}}\circ \mathfrak{X}\circ  i_\partial^{\#}	  &=f^{\#}|_{\L} \circ i_\partial^{\#}  \\
															&=f^{\#} \circ \begin{cases}
								i_\partial^{\#} 	&,~\text{in degree }\geq k \\
								0							&,~\text{in degree }<k
								\end{cases}		\\
&=			f^{\#} \circ i_\partial^{\#} \circ incl	\\
&=			g^{\#} \circ incl		
\end{align*}
where we used the identity $i_0^{\tilde{\#}} \circ \mathfrak{X} = f^{\#}|_{\L}$ from Lemma \ref{gamma}, $f^{\#}|_{\L} =f^{\#}$ in degrees $\geq k$, $f^{\#}|_{\L} =0$ in degrees $<k$, and finally 
\[
i_\partial^{\#} \circ incl=\begin{cases}
								i_\partial^{\#} 	&,~\text{in degree } \geq k \\
								0							&,~\text{in degree }<k
								\end{cases}		 
\]
by the construction of $\Rb$. 

As before we work with products such that $g^{\#}$ and $i_0^{\tilde{\#}}$ simultaneously are DGA homomorphisms. In the proof of Theorem \ref{thm:gDGAhom} we also obtained a cellular diagonal approximation on $L$ such that $i_\partial^\#$ is multiplicative. Restricting the induced product to $\L$ and restricting the product on $\C{\bar{X}}$ to $\bar{C}^\bullet$ makes the map $incl$ and $i_\partial^{\#}:\bar{C}^\bullet \to \L$ multiplicative. From the proof of Lemma \ref{gamma} it is clear that $\mathfrak{X}$ is multiplicative independent of the products we choose on $\C{\tL}$ and $\rC{\cL}$. Therefore we have a consistent choice of products. 
\end{proof}

\subsection{The intersection space cohomology de Rham map $\phi$}\label{sect:formphi}

We combine the maps $\tilde{\rho} :\WI \to \Rb$, $\phi_2 : \Rb \to Q^\bullet$ and $\phi_1^{-1} :  Q^\bullet \to \rC{\I} $ obtained in Section \ref{sect:deRhamPart}, \ref{sect:middlepart} and \ref{sect:pullback}, respectively, into our intersection space cohomology de Rham map $\phi:=\phi_1^{-1} \circ \phi_2 \circ \tilde{\rho}$. The DGA homomorphism $\phi_2$ induces a multiplicative maps already on the level of representatives and  $\phi_1^{-1}$ and  $\tilde{\rho}$ induce multiplicative maps on cohomology. Therefore, $\phi$ induces a multiplicative map on cohomology. In Section \ref{qisphi}, we check that this induced map is indeed is an isomorphism. But first let us write down the explicit form of $\phi$.

Note that $\phi_1^{-1}:Q^{\cdot}\to \rC{\I}$ is explicitly given by 
\[
\phi_1^{-1}(a,(i_\partial^{\#}a,b))=(a,b)
\]
and by the pullback construction we have, 
\[
\phi_2 (a)=(incl(a),\mathfrak{X} \circ  i_\partial^{\#}(a)).
\]
Further recall that
\[
\mathfrak{X}    
                =\begin{cases}
                        0 & ,~  deg\neq k\\
                        (id_{C^k(\tL)}\oplus \mathfrak{x})\circ f^{\# ,k}|_{\L} & ,~ deg=k ~ .
\end{cases}
\]
Together we have
\[
\phi_1^{-1}\circ\phi_2(a)\nonumber=\begin{cases}
                        (incl(a),0)&,~  deg(a)\neq k\\
                        (incl(a), \mathfrak{x}\circ f^{\# ,k}|_{\L}  \circ i_\partial^*(a))&,~  deg(a)=k
\end{cases}
\] 
and finally using that $incl \circ \tilde{\rho}=\rho_{\bar{X}}|$ by construction we arive at
\begin{align}
\phi(\omega)    &=\phi_1^{-1}\circ\phi_2 \circ \tilde{\rho} (\omega)\nonumber\\
                &=\begin{cases}
                        (\rho_{\bar{X}}| (\omega),0)&,~  deg(a)\neq k\\
                        (\rho_{\bar{X}}|  (\omega), \mathfrak{x}\circ f^{\# ,k}|_{\L}  \circ i_\partial^*\circ \tilde{\rho} (\omega))&,~  deg(a)=k~.
\end{cases} \label{phiuse}
\end{align}

\subsection{$\phi$ is a Quasi-Isomorphism}\label{qisphi}

To establish that $\phi$ is a quasi-isomorphism we make use of a 5-Lemma argument that is similar to the one given by Banagl in \cite[Section 9]{HIdR}. On the de Rham side the long exact sequence is induced by the following short exact sequence. 

\begin{lem}\cite[Lemma 9.5 adapted to our definition of $\WI$]{HIdR} \label{SES} \\ 
The sequence
\[
\begin{tikzcd}
0\rar&\WI \rar{incl.}& \WX \rar{proj. \circ i_{\partial}^{\#}} &\tW \rar&0
\end{tikzcd}
\]
is exact. Here $\tW$ is the orthogonal complement of $\TW$ with respect to the Hodge inner product on $\WL$, $incl.: \WI \to \WX$ is the subcomplex inclusion and $proj.: \WL\to \tW$ is orthogonal projection onto $\tW$. 
\end{lem} 
\begin{proof}
Certainly the inclusion of $\WI$ in $\WX$ is injective and thus exactness holds at $\WI$. The maps $i_\partial^{\#}$ and $proj.$ are surjective and so is their composition. We further observe that $i_\partial^{\#} \circ incl.(\WI) \subset \TW$ by construction of $\WI$ and $proj.\circ i_\partial^{\#} \circ incl.=0$ since $proj.$ is the projection to the orthogonal complement of $\TW$. This gives the exactness in the middle entry. In conclusion the sequence is exact. 
\end{proof}
Now recall that if we equip $\cL$ with the canonical cell structure induced from the cell structure of $\tL$ and choose the tip of the cone as base point, then
\[
\rC{\I}=\C{\bar{X}}\oplus C^{\cdot -1}(\tL)~.
\]

This orthogonal decomposition gives rise to the short exact sequence involving the inclusion of the second summand followed by the projection to the first $\pi_1$. Note that changing the sign of the first map preserves the exactness, therefore we have the short exact sequence 
\begin{equation}\label{eq:SESCell}
\begin{tikzcd}
0\rar&C^{\cdot -1}(\tL) \rar{I}& \rC{\cL} \rar{\pi_1} &\C{\bar{X}} \rar&0
\end{tikzcd}
\end{equation}
with
\begin{align*}
I: C^{\cdot -1}(\tL) &\to \C{\bar{X}} \oplus C^{\cdot -1}(\tL)=\rC{\I}\\
b &\mapsto -(0,b).
\end{align*}
We introduce the sign here since it will be necessary to have commutativity later. Let us remark that while $\pi_1$ is a cochain map, $I$ is only a cochain map up to sign, however this is enough to induce a long exact sequence on cohomology.

This sequences combines with the long exact sequence induced from the short exact sequences of Lemma \ref{SES} into the following diagram. 
\begin{equation} \label{5Lemmadiagram}
\begin{tikzcd}[column sep=tiny]
\Hpp{\WX} \rar \dar{\rho_{\bar{X}}^*}&
\Hpp{\tW} \rar\dar{f^*\circ \rho_L|^*}&
\Hp{\WI} \rar\dar{\phi^*}& 
\Hp{\WX} \dar{\rho_{\bar{X}}^*}
\\
\Hpp{\C{\bar{X}}}\rar&
\Hpp{\C{\tL}}\rar&
\Hp{\rC{\I}}\rar&
\Hp{\C{\bar{X}}}
\end{tikzcd}
\end{equation}
 Here $\rho_L|$ is the restriction of $\rho_L$ to $\tW$.
 
Let us have a look at the maps in this diagram. In the upper row we have (from left to right) the induced map of $proj.\circ i_\partial^{\#}$ and the connecting homomorphism $\delta$ obtained by the zig-zag construction and the induced map of $incl.$. Explicitly $\delta$ maps a cohomology class represented by a closed form $\omega$ in $\tW$ to the cohomology class represented by $d\bar{\omega}$. Here $\bar{\omega}$ is an extension of $\omega$ to $\bar{X}$ (i.e. $i_\partial^{\#}(\bar{\omega})=\omega$). This implies
\begin{equation}\label{calcideldomegabar}
i_\partial^{\#}(d\bar{\omega})=d i_\partial^{\#}(\bar{\omega})=d \omega=0
\end{equation}
since $\omega$ is closed. The middle arrow in the lower row is induced by $I$ and the right one by $\pi_1$. The connecting homomorphism on the right is given by $g^*$. To see this recall that by construction the connecting homomorphism maps a closed cohomology class represented by a closed cochain $x\in \C{\bar{X}}$ to the cohomology class represented by a cochain $l\in \C{\tL}$ such that $I(l)=d(x,l')$ with $l'\in C^{\cdot -1}(\tL)$ arbitrary, so
\[
(0,l)=I(l)=d(x,l')=(dx,g^\#*(x)-dl')=(0,g^\#(x)-dl')~.
\]
Therefore we have for the cohomology class of $l$
\[
[l]=[g^*(x)-dl']=[g^\# (x)]=g^*[x]~.
\]

We want to apply the 5-Lemma to diagram (\ref{5Lemmadiagram}) so we check the pre-requisites.
\begin{lem}\label{lemma:phiqis}
Diagram (\ref{5Lemmadiagram}) commutes at least up to a sign and has exact rows.
\end{lem}
\begin{proof}

The top and bottom rows are exact since they are part of the long exact sequence induced from the short exact sequence from Lemma \ref{SES} and the Sequence \ref{eq:SESCell}, respectively. In the following we prove the commutativity of this diagram. 

Let us start with the commutativity of the left square. If $q-1$ is greater or equal to $k,$ the cohomology group $\Hpp{\C{\tL}}$ vanishes and the square commutes trivially. If, however, $q-1$ is smaller than $k$ the projection of $\WL$ onto $\tW$ is the identity and we calculate explicitly
\begin{align*}
g^*\circ \rho_{\bar{X}}^*	&= f^* \circ i_\partial^* \circ \rho_{\bar{X}}^*\\
						&= f^* \circ \rho_L^* \circ i_\partial^*\\
						&= f^* \circ \rho_L^* \circ proj.^* \circ i_\partial^* .
\end{align*}

Next, we consider the middle square. For $q-1$ greater or equal to $k$ the cohomology group $\Hpp{\tW}$ vanishes and the square commutes for trivial reasons. Making use of Formula (\ref{calcideldomegabar}), we calculate
\begin{align*}
\phi^* \circ \delta [\omega]	  &= [\phi (d\bar{\omega})]
								\\&= \begin{cases}
                        [(\rho_{\bar{X}}|(d\bar{\omega}),0)]&,~ deg(\omega)\neq k\\
                        [(\rho_{\bar{X}}|(d\bar{\omega}), \mathfrak{x} \circ f^{\# ,k}|_\L \circ \tilde{\rho}_L \circ i_\partial^{\#}(d\bar{\omega}))]&,~ deg(\omega)=k
                        		\end{cases}
                        		\\&= [(\rho_{\bar{X}}|(d\bar{\omega}),0)].
\end{align*}
Since we equipped $\I$ with the canonical cell structure and choose the tip of the cone as base point its reduced chain complex is
\[
\widetilde{C}_\bullet(\I)=\Cc{\bar{X}} \oplus C_{\cdot -1}(\tL)
\]
with differential
\[
{\left(\begin{smallmatrix} \partial_{\bar{X}} & g_{\#} \\ 0 & -\partial_{\tL}[-1]   \end{smallmatrix} \right)}.
\]
Now take any reduced cellular $q$ cycle $(x,l)$ of $\I$. Being a cycle implies $\partial x= -g_{\#} l$. Further we follow the sign convention that for a $q-1$ cochain $\alpha$ and a $q$ chain $a$ we have
\[
d\alpha(a)=(-1)^q \alpha(\partial a).
\]
We use this to calculate 
\begin{align*}
[(\rho_{\bar{X}}|(d\bar{\omega}),0)]([x,l])	  &= \rho_{\bar{X}}|(d \bar{\omega})(x)
											\\&= d\rho_{\bar{X}}|(\bar{\omega})(x)
											\\&=(-1)^q \rho_{\bar{X}}|(\bar{\omega})(\partial x)
											\\&=(-1)^q \rho_{\bar{X}}|(\bar{\omega})(-g_{\#}l)
											\\&=(-1)^{q+1} g^{\#} \circ \rho_{\bar{X}}|(\bar{\omega})(l)																											\\&=(-1)^{q+1} f^{\#}\circ i_\partial^{\#} \circ \rho_{\bar{X}}|(\bar{\omega})(l)
											\\&=(-1)^{q+1} f^{\#}\circ \rho_{L} \circ i_\partial^{\#} (\bar{\omega})(l)
											\\&=(-1)^{q+1} f^{\#}\circ \rho_{L}(\omega)(l)
											\\&=(-1)^{q+1} (0,f^{\#}\circ \rho_{L}|(\omega))(x,l)
											\\&=(-1)^q I\circ f^{\#}\circ \rho_{L}|(\omega)(x,l)
											\\&=(-1)^q I^*\circ f^*\circ \rho_{L}|^* [\omega]([x,l]).
\end{align*}
Thus the middle square is commutative up to a sign.

Finally, we consider the right square.
\begin{align*}
\pi_1 \circ \phi (\omega)	  &= \begin{cases}
                        \pi_1 \circ(\rho_{\bar{X}}|(\omega),0)& deg(\omega)\neq k\\
                        \pi_1 \circ (\rho_{\bar{X}}|(\omega),\mathfrak{x} \circ f^{\# ,k}|_\L \circ \tilde{\rho}_L \circ i_\partial^{\#}(\omega))& deg(\omega)=k
\end{cases}
						\\&=\rho_{\bar{X}}|(\omega)
						\\&=\rho_{\bar{X}}\circ incl.(\omega)
\end{align*} 
This proves the commutativity of the square already on cochain level. 
\end{proof}

Furthermore, $\rho_L$ is a quasi-isomorphism and $f$ induces an isomorphism on cohomology in degrees lower than $k$. Thus their composition also induces an isomorphism in degrees lower than $k$. Restricting to the truncated complex $\tW$ yields a quasi-isomorphism $f^{\#}\circ \rho_L|$. The classical de Rham map also induces an isomorphism on cohomology.
In conclusion, the 5-Lemma is applicable and 
\[ \phi^*:\HH{\WI}\to \HH{\rC{\I}} \] 
is an isomorphism. We established before that $\phi^*$ is multiplicative and thus have proven our main result.

\Main

\section{Compatibility with Banagl's de Rham Theorem for HI}\label{Compatibility}
In \cite[Section 9]{HIdR}, Banagl constructs an alternative de Rham map, which is defined by integrating forms in $\WI$ over smooth cycles on the blowup $\bar{X}$ of $X$. 
We recall his construction and show that his de Rham map is compatible with the de Rham ring isomorphism of Section \ref{section_multdR}. 

Banagl uses a partial smooth model $\left( S_\bullet^{\propto} (g), \partial \right)$ for the mapping cone of the map $g: t_{<k} L \rightarrow \bar{X}.$ This chain complex is defined as $S_r^{\propto} (g) := S_r^{\infty} \left( \bar{X} \right) \oplus H_{r-1} (t_{<k} L)$ with the boundary operator $\partial: S_r^{\propto} (g) \to S_{r-1}^{\propto} (g)$ involving Lee's smoothing operator $s: S_\bullet (\bar{X}) \to S_\bullet^{\infty} (\bar{X})$, defined in \cite[Section 18]{Lee}, and a map $q: H_\bullet (t_{<k} L) \to S_\bullet (t_{<k} L),$ defined as follows.
For any $r \in \ZZ$, choose a completion $H_r'$ of $\im \, \partial_{r+1}$ in $\ker \, \partial_r$, i.e. $\im \, \partial_{r+1} \oplus H_r' = \ker \, \partial_r.$ Then, choosing a representative in $H_r'$ for each homology class $x \in H_r (t_{<k} L)$ gives rise to a map $q: H_{\bullet} (t_{<k} L) \to H_\bullet' \hookrightarrow \ker \partial \hookrightarrow S_\bullet (t_{<k} L),$ which satisfies $[q(x)] = x \in H_\bullet (t_{<k} L).$ Since the definition of the multiplicative de Rham isomorphism in this paper makes use of normalized singular chains, we choose $q$ such that its image is contained in the normalized chains. This is possible since the subcomplex inclusion from the complex of normalized singular chains to all singular chains is a quasi-isomorphism.

The boundary operator $\partial: S_r^{\propto} (g) \to S_{r-1}^{\propto} (g)$ is then defined as $\partial (v,x) := (\partial v + s g_{\#} q(x), 0)$. The partial smooth model $ S_\bullet^{\propto} (g)$ is quasi-isomorphic to the relative cochain complex $\widetilde{C}_\bullet (\I)$ by \cite[Proposition 9.2]{HIdR}. Since we want to relate the two different de Rham isomorphisms to each other, we want to give an explicit description of the chain maps that induce this homology isomorphism.
Therefore, let $\bar{S}_\bullet (g)$ be the complex defined by $\bar{S}_r (g) := \widehat{S}_r (\bar{X}) \oplus  H_{r-1} \left(t_{<k} L \right)$ with $\partial (v,x) = (\partial v + g_{\#} q(x), 0 )$, where $\widehat{S}_\bullet (\bar{X})$ denotes the chain complex of nondegenerate singular chains. It fits into the diagram of quasi-isomorphisms
\[
  \widetilde{C}_\bullet (\I) \xleftarrow{\gamma_{\bar{X}{\#}} \oplus \gamma_{t_{<k} L{\#}} q} \bar{S}_\bullet (g) \xrightarrow{\widehat{s} \oplus \text{id}} S_\bullet^{\propto} (g).
\]
Here, we use the identification $\widetilde{C}_\bullet(\I)=\Cc{\bar{X}} \oplus C_{\cdot -1}(\tL)$ as in the proof of Lemma \ref{lemma:phiqis} and let $\widehat{s}: \widehat{S}_\bullet (\bar{X}) \hookrightarrow S_\bullet (\bar{X}) \xrightarrow{s} S_\bullet^\infty (\bar{X})$ be the composition of the denoted subcomplex inclusion and Lee's smoothing operator.
As in Section \ref{cellulardR}, the maps $\gamma$ denote homotopy equivalences coming with the geometric realization.
The following diagram
\[
\begin{tikzcd}
  \widehat{S}_\bullet \left( t_{<k} L \right) \ar{r}{g_{\#}} \ar{d}{\gamma_{t_{<k} L{\#}}} & \widehat{S}_\bullet (\bar{X}) \ar{d}{\gamma_{\bar{X}{\#}}} \\
C_\bullet \left( t_{<k} L \right) \ar{r}{g_{\#}} & C_\bullet (\bar{X})
\end{tikzcd}
\]
commutes as a special case of a more general commutative diagram that contains the $\gamma_{\#}$'s and any cellular continuous map between $t_{<k} L$ and $\bar{X}$. We show that $\gamma_{\bar{X}{\#}} \oplus \gamma_{t_{<k} L{\#}} q$ is a chain map with the following calculation,
\[ 
\begin{split}
  & \partial \left( \gamma_{\bar{X}{\#}} v,  \gamma_{t_{<k} L{\#}} q(x) \right) = \bigl ( g_{\#} \gamma_{t_{<k} L{\#}} q (x) - \partial (\gamma_{\bar{X}{\#}} v), \underbrace{\partial \gamma_{t_{<k} L{\#}} q (x)}_{=\gamma_{t_{<k} L{\#}} \left ( \partial q (x)\right) } \bigr) \\
  = ~& \left( g_{\#} \gamma_{t_{<k} L{\#}} q (x) - \partial (\gamma_{\bar{X}{\#}} v), 0 \right)
  = (\gamma_{\bar{X}{\#}} g_{\#} q (x) - \gamma_{\bar{X}{\#}} \partial v, 0 ) \\ 
	= ~& \gamma_{\bar{X}{\#}} \oplus \gamma_{t_{<k} L{\#}} q  \left( \partial (x,v) \right).
\end{split}
\]
It is a quasi-isomorphism by the following argument: All the mapping cone-like complexes fit into short exact sequences with the complexes contained in the cone on the left and right. The maps $\gamma_{t_{<k} L{\#}} q, \gamma_{\bar{X}{\#}}$ and their direct sum fit into a diagram of these two short exact sequences. Since $\gamma_{t_{<k} L{\#}} q$ and $ \gamma_{\bar{X}{\#}}$ are quasi-isomorphisms, the 5-Lemma gives that their direct sum is also a quasi-isomorphism.
The map $\widehat{s} \oplus \text{id}$ is a quasi-isomorphism by the same argument. Note, that this argument can also be used to prove \cite[Lemma 9.1]{HIdR}.

Since the integral of any smooth differential form over a degenerate simplex vanishes, the de Rham maps are indifferent to the use of normalized or non-normalized singular simplices, so we will neglect this distinction in the following. Banagl's de Rham map is defined as
\[ 
\begin{split}
\phi_B: H^\bullet \left( \WI \right) &\rightarrow H_\bullet \left( S_\bullet^{\propto} (g) \right)^{\dagger}, \\
\phi_B \left( [\omega] \right) \left( \left[ (v,x) \right] \right) &:= \int_v \omega~.
\end{split}
\]
It is noted as $\Psi_{\bar{p}}$ by Banagl, but we call it $\phi_B$ to be consistent with our previous notation of the de Rham morphisms for intersection space cohomology. $\phi_B$ is an isomorphism, as is shown in \cite[Theorem 9.11]{HIdR}. We adapt this map to the intermediate complex $\bar{S}_\bullet (g)$. 
\[ \begin{split} 
\bar{\phi}_B: H^\bullet \left( \WI \right) & \rightarrow H_r\left( \bar{S}_\bullet (g) \right)^\dagger, \\
\bar{\phi}_B \left( [\omega] \right) \left( \left[ (v,x) \right] \right) & := \int_{sv} \omega.
\end{split} \] 
This map is well defined by the same arguments as in \cite[Prop. 9.8]{HIdR} and fits into the following commutative diagram of isomorphisms
\begin{equation}\label{diag:phibar}
\begin{tikzcd}
  H^\bullet \left( \WI \right) \ar{r}{\phi_B} \ar{rd}[swap]{\bar{\phi}_{B}} & H_\bullet \left( S_\bullet^\propto (g) \right)^\dagger \ar{d}{(\widehat{s} \oplus \text{id})^\dagger} \\
\ & H_\bullet \left( \bar{S}_\bullet (g) \right)^\dagger~.
\end{tikzcd}
\end{equation}

The choice made to define the partial smooth complex $S_\bullet^\propto (g)$ corresponds to a choice for the map $\mathfrak{x}: C^k \left( t_{<k} L \right) \rightarrow C^{k-1} \left( t_{<k} L \right)$ defined in Lemma \ref{gamma}, which was used to define $\phi$.
We employ the following choice. First, choose a basis $ \{x_1,\cdots,x_r \} $ of $C_k (t_{<k} L). $ The Moore approximation or spatial homology truncation, defined in \cite[Chapter 1.1]{intspaces}, is installed such that the boundary map $\partial_k: C_k(t_{<k} L) \rightarrow C_{k-1}(t_{<k} L)$ is injective, hence $ \mathcal{B}:= \left\{ \partial_k x_1,\cdots, \partial_k x_r \right\} $ is a basis for $ \text{im} \, \partial_k \subset C_{k-1} \left( t_{<k} L \right).$ The morphism $q: H_{k-1} \left( t_{<k} L \right) \rightarrow H_{k-1}'$ maps a basis $ \left\{ \xi_1,\cdots,\xi_l \right\} $ of $H_{k-1} \left( t_{<k} L \right)$ to the basis $\left\{ q(\xi_1),\cdots, q(\xi_l) \right\} $ of $H_{k-1}'.$ Since $\gamma_{\#}: \widehat{S}_\bullet \left( t_{<k} L \right) \rightarrow C_\bullet \left( t_{<k} L \right)$ is a quasi-isomorphism, the set $\mathcal{J}:= \left\{ \gamma_{\#} \left( q \xi_1 \right),\cdots, \gamma_{\#} \left( q \xi_l \right) \right\}$ completes $ \mathcal{B} $ to a basis of $\ker \partial_{k-1}.$ Now choose any completion $ \mathcal{B} \cup \mathcal{J} \cup \left\{ y_1,\cdots, y_s \right\}$ to a basis of $C_{k-1} \left( t_{<k} L \right) $ and let $ \left\{ (\partial_k x_i)^\dagger, \left( \gamma_{\#} (q \xi_j) \right)^\dagger, y_k^\dagger \right\}_{i,j,k} $ be the corresponding dual basis of $C^{k-1} \left( t_{<k} L \right).$ Then, for any $1 \leq i \leq r,$ $1\leq j \leq l,$ and $1\leq k\leq s$, we get 
\[ d \left( \gamma_{\#} (q \xi_j) \right)^\dagger (x_i) = \left( \gamma_{\#} (q \xi_j)\right)^\dagger (\partial_k x_i) = 0 \]
as well as
\[ d ( y_k^\dagger ) (x_i) = y_k^\dagger (\partial_k x_i) = 0. \]
In other words, all the cochains $\left( \gamma_{\#} (q \phi_i) \right)^\dagger, y_j^\dagger$ are closed.
The cochains $(\partial_k x_i)^\dagger$ are not closed since $d \left( (\partial_k x_i)^\dagger \right) (x_i) = 1.$ To deduce the desired compatibility result for the different de Rham maps, we choose the map $\mathfrak{x}: C^k \left( t_{<k} L \right) \rightarrow C^{k-1} \left( t_{<k} L \right) $ such that its image is contained in the span of $ \left\{ (\partial_k x_1)^\dagger, \cdots, (\partial_k x_r)^\dagger \right\}$.

\begin{rmk}
Before stating the compatibility theorem, note that we do not write all the decorations of the different geometric realizations $\gamma$ in the following. We do so to make the theorem and proof more readable. We encourage the reader to check which $\gamma$ is used in the respective situation. For the sake of readability we will in the following also write $\phi$ for the map $\phi^*$.
\end{rmk}

\begin{thm}\label{thm:Compatability}
Let the map $\mathfrak{x}: C^k \left( t_{<k} L \right) \rightarrow C^{k-1} \left( t_{<k} L \right) $ be chosen as described above.
Then the following diagram of isomorphisms commutes
\[
\begin{tikzcd}
 H^\bullet \left( \WI \right) \ar{d}{\cong}[swap]{\phi} \ar{r}{\bar{\phi}_B}[swap]{\cong} & H_\bullet \left( \bar{S}_\bullet (g) \right)^\dagger \\
 \widetilde{H}^\bullet \left( \I \right) \ar{r}{\cong} & 
 \widetilde{H}_\bullet \left(\I \right)^\dagger \ar{u}{\cong}[swap]{(\gamma_{\#} q \oplus \gamma_{\#})^\dagger} ~.
\end{tikzcd}
\]
Here, $\widetilde{H}^\bullet \left( \I \right) \rightarrow \widetilde{H}_\bullet \left( \I \right)^\dagger $ is the standard map, that is induced by evaluating any representative of a cellular cohomology class on any representative of a homology class.
\end{thm}
\begin{proof}
  Let $ \omega \in \WIp$ and let $(v,x) \in \bar{S}_p (g)$ be closed. Then 
  \[ \bar{\phi}_B \left( [\omega] \right) \left( \left[ (v,x) \right] \right) = \int_{sv} \omega. \]
On the other hand, equation (\ref{phiuse}) yields 
\begin{align*}
  &\phi \left( [\omega] \right) \left( \left[ (\gamma_{\#} v, \gamma_{\#} qx)  \right] \right) \\
&=\begin{cases}
 0 + \rho_{\bar{X}}| (\omega) (\gamma_{\#} v) 	&,~p\neq k\\
\left (\mathfrak{x} \left( f^k|_\L \circ \tilde{\rho}_L \circ i_\partial^{\#}(\omega)\right) \right ) (\gamma_{\#} qx)+ \left( \rho_{\bar{X}}|(\omega) \right) (\gamma_{\#}v)		&,~p=k~.
\end{cases}
\end{align*}
To be precise, the map $\phi$ denoted here is the composition of the actual de Rham map $\phi$ with the isomorphism $\widetilde{H}^\bullet \left( \I \right) \xrightarrow{\cong} \widetilde{H}_\bullet \left( \I \right)^{\dagger}.$
Note that
\[ \left( \mathfrak{x} \left( f^k|_\L \circ \tilde{\rho}_L \circ i_\partial^{\#} (\omega) \right) \right) \in \text{span} \left( \left\{ \left( \partial_k x_1 \right)^\dagger, \cdots, \left( \partial_k x_r \right)^\dagger \right\} \right), \]
 by our choice for the map $\mathfrak{x}$.
Our choice of basis for $C_{k-1} \left( t_{<k} L \right)$ implies
\[ \left( \mathfrak{x} \left( f^k|_\L \circ \tilde{\rho}_L \circ i_\partial^{\#} (\omega) \right) \right) \left( \gamma_{\#} q x \right) = 0, \]
since it is obviously true, that $\gamma_{\#} qx \in \text{span} \left\{ \gamma_{\#}(q \xi_1), \cdots, \gamma_{\#}(q \xi_l) \right\}.$
The result is the following equality.
\[
  \phi \left( [\omega] \right) \left( \left[ (\gamma_{\#} v, \gamma_{\#} qx) \right] \right) = \rho_{\bar{X}}| (\omega) (\gamma_{\#} v). 
\]

As we mentioned in Section \ref{cellulardR}, the map $\gamma_{\bar{X}}: \Gamma \bar{X} \to \bar{X}$ is a homotopy equivalence, which maps the geometric realization of the boundary $i_\partial: \partial \bar{X} \hookrightarrow \bar{X}$ to this boundary, $\gamma_{\bar{X}}|: \Gamma \partial \bar{X} \to \partial \bar{X}$. This restriction is also a homotopy equivalence. Let $\delta_{\bar{X}}: \bar{X} \to \Gamma \bar{X}$ be a homotopy inverse of $\gamma_{\bar{X}}$, which maps the boundary $\partial \bar{X}$ to its realization $\Gamma \partial \bar{X}$ and $\bar{H}: \Gamma \bar{X} \to \Gamma \bar{X}$ a homotopy between $\delta_{\bar{X}} \gamma_{\bar{X}}$ and the identity $\id_{\Gamma \bar{X}}$. Such $\delta$ and $H$ exist by the homotopy extension and lifting property, see \cite[Chapter 10.3]{May}. 
Since the cellular chain complex of the geometric realization is naturally isomorphic to the singular chain complex of $\bar{X}$, this homotopy $\bar{H}$ induces the chain homotopy $H: \widehat{S}_{\bullet} (\bar{X}) \to \widehat{S}_{\bullet +1} (\bar{X})$ between the chain map $\delta_{\bar{X}{\#}} \gamma_{\bar{X}{\#}}: \widehat{S}_\bullet (\bar{X}) \rightarrow \widehat{S}_\bullet (\bar{X})$ and the identity. The restriction $\bar{H}|$ to $\Gamma \partial \bar{X}$ induces a chain homotopy $H_\partial: \widehat{S}_{\bullet} (\partial \bar{X} ) \to \widehat{S}_{\bullet +1} (\partial \bar{X})$ between $\delta_{\bar{X}}|_\# \circ \gamma_{\bar{X}}|_\#: \widehat{S}_\bullet (\partial \bar{X}) \to \widehat{S}_\bullet (\partial \bar{X})$ and the identity such that 
	$ {i_{\partial}}_\# \circ H_\partial = H_\partial \circ {i_\partial}_\#.$

 We then get the following equation, using the definition of the absolute de Rham map $\rho_{\bar{X}}$ of Section \ref{cellulardR} (where it is denoted by just $\rho$).
\[ 
\begin{split}
	\rho_{\bar{X}}| (\omega) (\gamma_{\bar{X}{\#}} v ) &= \left( \delta_{\bar{X}}^{\#} s^\dagger \rho_{s} (\omega) \right) (\gamma_{\bar{X}{\#}} v) = \rho_{s} (\omega) \left( s \delta_{\bar{X}{\#}} \gamma_{\bar{X}{\#}} v \right) \\
&= \rho_{s} (\omega) \left( s (v + \partial H v + H (\underbrace{\partial v}_{=- g_{\#} qx}) \right) \\
&= \int_{sv} \omega + \int_{s \partial Hv} \omega - \int_{s Hg_{\#} qx} \omega = \int_{sv} \omega,
\end{split}
\]
We made use of $(v,x) \in \bar{S}_p (g)$ being a cycle as well as of the facts, that
\[ \int_{s \partial Hv} \omega = \int_{sHv} d\omega = 0, \]
and (using $g = i_\partial f$ and $H \circ {i_\partial}_\# = {i_\partial}_\# \circ H_\partial$) 
\[ \int_{s H g_{\#} q x} \omega = \int_{s {i_\partial}_{\#} H_\partial (f_{\#}qx) } \omega = \int_{s H_\partial (f_{\#} qx)} i_\partial^{\#} \omega =0, \]
because $qx = 0$ for $p \geq k,$ while $ i_\partial^{\#} \omega = 0$ for $p <k.$ The conclusion is the statement of the theorem,
\[
  \phi \left( [\omega] \right) \left( \left[ (\gamma_{\#} v, \gamma_{\#} qx) \right] \right) = \int_{sv} \omega = \bar{\phi}_B \left( [\omega] \right) \left( \left[ (v,x) \right] \right)~.
\]
\end{proof}

\noindent
Finally, we combine Theorem \ref{thm:Compatability} with Diagram (\ref{diag:phibar}) into the following commutative diagram.
\[
\begin{tikzcd}
\			&	H^\bullet \left( \WI \right) \ar{rd}{\phi_B} \ar{d}{\bar{\phi}_{B}} \arrow[ld,"\phi" above] 	&	\	\\
\widetilde{H}^\bullet \left( \I \right) \ar{r}{\cong} 		& H_\bullet \left( \bar{S}_\bullet (g) \right)^\dagger & H_\bullet \left( S_\bullet^\propto (g) \right)^\dagger \arrow[l,"\cong" above] ~.
\end{tikzcd}
\]
\begin{rmk}
Note, that only $H^\bullet \left( \WI \right)$ and $\widetilde{H}^\bullet \left( \I \right)$ are naturally equipped with a cup product. The question, if any of the maps besides $\phi$ are ring isomorphisms, therefore depends on our choice of product on $H_\bullet \left( \bar{S}_\bullet (g) \right)^\dagger$ and $H_\bullet \left( S_\bullet^\propto (g) \right)^\dagger$. Since all the maps in the diagram are isomorphisms, they are multiplicative if and only if we choose the products that are obtained by transporting the naturally defined products via those isomorphisms. All products defined in this way are consistent due to the commutativity of the diagram.

This means that Banagl's de Rham isomorphism $\phi_B$ is a ring isomorphism if and only if we define the cup product $\cup_{\propto}$ on $H_\bullet (S^\propto_\bullet (g) )$ as the transport of the cup product $\cup_{\I}$ of $\widetilde{H}^\bullet \left( \I \right)$ via the isomorphism $\mathcal{I}:H_\bullet \left( S_\bullet^\propto (g) \right)^\dagger \xrightarrow{\cong} \widetilde{H}^\bullet (\I)$, i.e.
\[ \alpha \cup_{\propto} \beta := \mathcal{I}^{-1} \left( \mathcal{I} (\alpha) \cup_{\I} \mathcal{I} (\beta) \right) \quad \text{for} ~ \alpha, \beta \in H_\bullet (S^\propto_\bullet (g) )^{\dagger}. \]
It is unclear though if this cup product on cohomology level comes from a cup product on the cochain complex $S_\propto^\bullet (g)$, defined in analogy to $S^\propto_\bullet (g)$. A cup product on the standard mapping cone $S^\bullet (g)$ can be defined by 
\[
  (\psi, \mu) \cup (\xi, \nu) := (\psi \cup \xi, \mu \cup g^\# \xi).
\]
The definition of $S_\propto^\bullet (g)$ uses the map $q$, however, which depends on choices and cannot be expected to be natural. Therefore, it is open whether this construction also gives a cup product on the cochain complex $S_\propto^\bullet (g)$.
\end{rmk}

\subsection*{Acknowledgements}
We thank Prof. Dr. Markus Banagl (Universität Heidelberg) for his input as supervisor of the master thesis of the first author on which parts of this article are based.
The second author wants to thank the Canon Foundation, that supported him during his stay at the Hokkaido University, Japan, and Prof. Toru Ohmoto for being a generous host.
\bibliography{bib}
\bibliographystyle{plain}

\end{document}